\documentclass[10pt]{article}
\usepackage{amssymb}
\usepackage{amsmath}
\usepackage{amsfonts}
\usepackage{amscd}
\usepackage{geometry}
\usepackage[all,cmtip]{xy}
\usepackage{makeidx}
\usepackage{titlesec}
\usepackage{fancyhdr}
\usepackage[colorlinks]{hyperref}
\usepackage [pdftex]{graphicx}

\usepackage{mathtools} % NUEVO
\usepackage{thmtools, thm-restate}  % nuevo
\usepackage [pdftex]{color}
\usepackage{cancel}
\usepackage[all]{xy}
\usepackage{tikz}
\usepackage{fancyhdr}
 \usepackage{latexsym}
\usepackage{geometry,enumerate,latexsym,tabularx}
\usepackage{multicol}
\usepackage{wrapfig}
\usepackage{stmaryrd}
\usepackage[cm]{fullpage}
\usepackage{layout}
\usepackage{eqnarray}
\usepackage{array}
\DeclarePairedDelimiter{\abs}{\lvert}{\rvert} % NUEVO

\newcommand{\C}{\mathbb{C}}
\newcommand{\R}{\mathbb{R}}
\newcommand{\Q}{\mathbb{Q}}

\newcommand{\Z}{\mathbb{Z}}

\DeclareMathOperator{\SO}{SO}
\DeclareMathOperator{\SU}{SU}

\def\C{\mathbb{C}}

\def\H{\mathbb{H}}

\def\Im{\mathrm{Im}}

\def\PSL{\mathbf{PSL}}
\def\R{\mathbb{R}}
\def\Re{\mathrm{Re}}
\def\S{\mathbb{S}}

\def\SL{\mathbf{SL}}
\def\SO{\mathbf{SO}}

\def\SU{\mathbf{SU}}

\def\Z{\mathbb{Z}}

\titleformat{\section}[hang]
{\normalfont\filright\large}{\thesection. }{0pt}
{\upshape\bfseries}

\titleformat{\subsection}[hang]
{\itshape}{\thesubsection \ - }{0pt}
{}

\usepackage{amsthm}
\theoremstyle{plain}
\newtheorem{theo}{Theorem}[section]
\newtheorem{prop}[theo]{Proposition}
\newtheorem{lemm}[theo]{Lemma}

\newtheorem{rem}[theo]{Remark}

\theoremstyle{remark}

\theoremstyle{definition}
\newtheorem{defi}[theo]{Definition}
\newtheorem*{Example-non}{Example}
\newtheorem*{Lemma-non}{Lemma}
\newtheorem*{Proposition-non}{Proposition}
\newtheorem*{Corollary-non}{Corollary}
\newtheorem*{Definition-non}{Definition}

\providecommand{\abs}[1]{\lvert#1\rvert}

\newcommand*\rfrac[2]{{}^{#1}\!/_{#2}}

\title{A note on Mellin transform, Eisenstein Series and distribution $d\epsilon_{it}$ on $\PSL(2,\Z[i]) \backslash \PSL(2,\C)$}
\author{\small Otto Romero \footnote{Departamento de Matem$\acute{\mbox{a}}$ticas, CINVESTAV, Apdo. Postal 14-470, 07000, CDMX, M\'exico. Email:  \texttt{ottohrg@gmail.com}}}

\date{\today}
\begin{document}

\maketitle

\begin{abstract}
Let  $\PSL(2,\Z[i])  \backslash \PSL(2,\C)$ be the frame bundle of the Picard variety  $\PSL(2,\Z[i]) \backslash \H^3$, $f$ a smooth function with compact support defined on $\PSL(2,\Z[i])  \backslash \PSL(2,\C)$ and  $\mathfrak{M}(f,s)$  the Mellin transform of $f$, in this note we prove in the proposition (\ref{PPr}) that
 \begin{equation}\label{PPO2}
 \mathfrak{M}(f,s)   \, = \,  \sum_{\substack{ l \in \Z \\ l \geq 0 }}  \;    \sum_{\substack{ k,m=-\rfrac{l}{2} \\ k,m \in \tfrac{l}{2} \Z \\ k \equiv m \equiv \rfrac{l}{2} \text{ mod } 1 }}^{\rfrac{l}{2}}   \, (-1)^{m-k} \cdot \mathcal{T}^{\, l}_{kk} (I)
                 \int_{F}    \widehat{f}_{k,m}^{\, l} ( z+\lambda j) \cdot  {E}_{-m,-k}^{\rfrac{l}{2}} (z +\lambda j,s) \; \frac{dx dy d\lambda}{\lambda^3},
\end{equation}
where  $\widehat{f}_{k,m}^{\, l} ( z+\lambda j)$ are the Fourier expansion coefficients of $f$ on the fibers, ${E}_{k,m}^{l}(g,s)$ are the Eisenstein series defined on $\SL(2,\C)$ using the representations of $\SU(2)$ (see formula (\ref{ES})) and $\mathcal{T}_{km}^L$ are a basis of  eigenfunctions for Laplacian defined on  $\SU(2)$ (see formula (\ref{T1})). The previous result is a generalization of the proposition 2.6 in  Sarnak \cite{Sar80} for functions in the unit tangent bundles of a hyperbolic surfaces. Using  (\ref{PPO2}) we can define the micro-local lift $ d\epsilon_{it}$ (to $\SL(2,\C)$)  of distribution $\abs{E \big(z+\lambda j, it \big) }^2  \, dV$ as follows:
\begin{equation}\label{Medida}
  d\epsilon_{it} \, := \,   E \big( z+\lambda j,  it \big)
 \;  \sum_{\substack{ L \in  \mathbb{Z} \\ L \geq 0 }}   \; \sum_{\substack{ k,m=-\rfrac{L}{2} \\  k,m \in \tfrac{l}{2} \Z \\ k \equiv m \equiv \rfrac{L}{2} \text{ mod } 1 }}^{\rfrac{L}{2}}  \, (-1)^{m-k} \cdot \mathcal{T}^{\, L}_{kk} (I)  \cdot  {E}_{-m,-k}^{\rfrac{L}{2}} \big(z+\lambda j, -it\big)  \cdot \overline{\mathcal{T}_{km}^L}.
\end{equation}
The distribution  $d\epsilon_{it}$ is analogous to the distribution $d\epsilon_{\rfrac{1}{2}+it}$ which appears in Zelditch \cite{Zel} for hyperbolic surfaces. We use  ideas from  Luo and Sarnak \cite{LuoSar}, Jakobson \cite{Jak} and Koyama \cite{Koy}  to  calculate  $(f,d\epsilon_{it})$ for  $f$ a cuspidal form (see formula \ref{FC18}) and for $f$ an incomplete  Eisenstein series (see formula \ref{IES25}). We also establish asymptotic estimates when $ t $ tends to $ \infty $ (propositions  (\ref{PFC}) and (\ref{IncEisSeries})).

\bigskip
We conjecture that the new positive distribution $d\epsilon_{it}^F$, constructed with the Friedrichs' symmetrization technique  applied to $ d\epsilon_{it}$, satisfies the same asymptotic estimates that $d\epsilon_{it}$ as in the propositions (\ref{PFC}) and (\ref {IncEisSeries}). This is what happens in the case of $\PSL(2,\Z) \backslash \PSL (2,\R) $, see Jakobson \cite {Jak}. This would imply that if  $\Omega_1$ and $\Omega_2$ ($\text{vol }(\Omega_2) \neq 0$) be  compact Jordan measurable sets in  $\PSL(2,\Z[i]) \backslash \PSL(2,\C)$ then
\begin{equation}\label{QUEPM}
 \lim_{t \rightarrow \infty} \frac{ \int_{\Omega_1} d\epsilon_{it}^F }{ \int_{\Omega_2} d\epsilon_{it}^F }  \, = \, \frac{\text{vol} \, (\Omega_1)}{\text{vol} \, (\Omega_2)}.
\end{equation}

This last assertion is  the quantum ergodicity for Eisenstein series on $\PSL(2,\Z[i]) \backslash \PSL(2,\C)$.


\end{abstract}

\setcounter{tocdepth}{5}
\tableofcontents

\section*{Introduction}
The idea of integrating an invariant function with an Eisenstein series was introduced by Rankin and Selberg. Let $\Sigma=\PSL(2,\Z) \backslash \H^2$ the modular surface and  $f \in C_0^\infty (\Sigma)$,  the Mellin transform of  $f$ in $s$ can be obtained by integrating $ f $ with the Eisenstein series $ E (z, s) $ associated with $ \PSL (2, \Z) $. If $f \in  C_0^\infty (S \Sigma)$, with  $S \Sigma$ denote the unit tangent bundle to  $\Sigma$,  Sarnak \cite{Sar80} gives a similar formula using the Eisenstein series $ E_n (g, s) $ defined on $ \SL (2, \R) $, see  Kubota \cite{Kub} chapter VI. 
 
 \begin{prop} (Sarnak) \label{PSar} Let  $f \in C^\infty_{0} ( S   \Sigma )$ ($f$  is a smooth complex function with compact support) and now we assume that $\Sigma$ is a hyperbolic surface of finite area with  only one cusp in infinity, $s \in \C$ such that  $\text{Re}(s)>1$. Consider the Fourier expansion of $f$ 
\[
 f(z,\theta)  \, = \, \sum_{ n \in \Z } \widehat{f}_n (z) \,  e^{i n \theta}. 
\] 
Then
\[
 \mathfrak{M}(f,s) \, = \,  \sum_{n=0}^\infty  \int_{D}    \widehat{f}_{n}  ( z ) \cdot  {E}_{2n}  (z,s) \; dw(z).
\]
%where $\mathfrak{M}(f,s)$ is the Mellin transform for the complex function $f$ and $F$ is a fundamental domain for $\Sigma$.
\end{prop}
%\bigskip
%In  \cite{Otto}  prove an analogous for functions on unit tangent bundle  $T_1 M_D=T_1 \mathbb{H}^3/{\Gamma^\ast_D}$ (the action is via the differential).

We now turn to dimension 3, let $ M = \PSL (2, \Z [i]) \backslash \H^3 $ be the Picard variety. In \cite{Otto} we prove  an analogue of the proposition (\ref{PSar}) for functions defined on $ S M $, where  $S M= {\PSL(2,Z[i])^\ast} \backslash S \H^3$ (the action is via the differential) and $S \H^3$ is the unit tangent bundle to $\H^3$.
It is then natural to extend the previous result to functions on $ \PSL(2, \Z[i]) \backslash \PSL (2, \C) $. The formula (\ref{PPO2}) corresponds to that case.

%\begin{prop}\label{PO1} Let   $f \in C^\infty_{0} \big( S M \big)$,  $\text{Re}(s)>1$. Then, 
%\[
% \mathfrak{M}(f,s)   \, = \, 
%       \sum_{\substack{ l \in \Z \\ l \geq 0 }} \;  \sum_{\substack{ k,m=-l \\ k,m \in %\Z \\ k \equiv m \equiv l \text{ mod } 1  }}^{l} \,  e^{-i (k+m) \pi}  \cdot  Y_m^l ( e_1 )      \int_{F}   \widehat{f}_k^{\, l} ( z+ \lambda j) \cdot E_{km}^l (z+\lambda j,s) \; \frac{dx dy d\lambda}{\lambda^3}.
% \]
%\end{prop}

%\bigskip

\bigskip
\bigskip
On the other hand, results of Zelditch's paper  \cite{Zel} imply that if $ \Sigma $ is the modular surface and $ \varphi_j $ the discrete spectrum of the hyperbolic Laplacian, that is, $\Delta \varphi_j = \lambda_j \varphi_j $ with $\Delta \, = \, y^2 \Big( \frac{\delta^2}{\delta x^2} + \frac{\delta^2}{\delta y^2} \Big)$, then if $\Omega$ ($\text{area} \, (\Omega) \neq 0$) is a compact Jordan measurable set in $X$, it exists a subsequence $\varphi_{j_k}$ such that 
$$ \lim_{k \rightarrow \infty} \int_{\Omega} \abs{\varphi_{j_k}}^2 \, dA \, = \, \frac{\text{area} \, (\Omega)}{\text{area} \, (X)}. $$

\bigskip
For the continuous spectrum of hyperbolic Laplacian given by the Eisenstein series, Luo and Sarnak  \cite {LuoSar} proved that for any  compact Jordan measurable sets  $\Omega_1$ and $\Omega_2$ ($\text{area} \, (\Omega_2) \neq 0$) in  $\Sigma$  is true that
\begin{equation}\label{Intro3}
 \lim_{t \rightarrow \infty} \frac{ \int_{\Omega_1} \abs{E \big(z, \tfrac{1}{2} +it \big) }^2  \,  \tfrac{dxdy}{y^2} }{ \int_{\Omega_2} \abs{E \big(z, \tfrac{1}{2} +it \big)  }^2  \,  \tfrac{dxdy}{y^2} }  \, = \, \frac{\text{area} \, (\Omega_1)}{\text{area} \, (\Omega_2)},
\end{equation}
which follows from
\begin{equation}\label{Intro4}
\int_{\Sigma} f(z) \cdot \abs{E \big(z, \tfrac{1}{2} +it \big)}^2  \,  \frac{dxdy}{y^2}  \, \sim \, \frac{48}{\pi} \, \ln t \; \int_{\Sigma} f(z) \,  \frac{dxdy}{y^2}.
\end{equation}
%where $F(z)$ is a continuous function of a compact support and $E(z,s)$ is an Eisenstein series for $\PSL(2,\Z)$.

\bigskip
Jakobson \cite{Jak} formulated an analog of (\ref{Intro4}) and then of (\ref{Intro3}) for $ \PSL(2,\Z) \backslash \PSL(2,\R) $, for this, it was required a micro-local lift of the positive distribution $ \abs{E \big (z, \tfrac{1}{2} + it \big)}^2 \, dA $ to $ S \Sigma $. Zelditch \cite{Zel} gave the following distribution:
\begin{equation}\label{DisZ}
 d\epsilon_{\rfrac{1}{2}+it} \, = \,   E \big( \cdot, \tfrac{1}{2} +it \big) \sum_{k=0}^\infty E_{2k} \big( \cdot, \tfrac{1}{2} - it \big) \, e^{-2ik \theta}.
\end{equation}

Zelditch noted that the distribution $ \epsilon_{\rfrac{1}{2} + it} $ is not positive. A new positive distribution $ d\epsilon_{\rfrac{1}{2} + it}^ F $ built using the Friedrichs' symmetrization technique is required, see Zelditch \cite{Zel}.
\begin{theo}[Jakobson]
 Let  $\Omega_1$ and $\Omega_2$ ($\text{vol} \, (\Omega_2) \neq 0$) be arbitrary compact Jordan measurable sets in \\$\PSL(2,\Z) \backslash \PSL(2,\R)$. Then 
\begin{equation}\label{Intro5}
 \lim_{t \rightarrow \infty} \frac{ \int_{\Omega_1} d\epsilon_{\rfrac{1}{2}+it}^F }{ \int_{\Omega_2} d\epsilon_{\rfrac{1}{2}+it}^F }  \, = \, \frac{\text{vol} \, (\Omega_1)}{\text{vol} \, (\Omega_2)}.
\end{equation}
\end{theo}

\bigskip
Koyama \cite{Koy} got the analog result at (\ref{Intro3}) for arithmetic 3-orbifolds $M_D = \PSL(2,\mathcal{O}_D)  \backslash \H^3$ with only one cusp, where  $\Q( \sqrt{-D})$  be a imaginary quadratic field and $\mathcal{O}_D$ its ring of integers, 
this includes the Picard variety that corresponds to the field $\Q( \sqrt{-1})$. 
\begin{theo}[Koyama]
 Let  $\Omega_1$ and $\Omega_2$ ($\text{vol} \, (\Omega_2) \neq 0$) be arbitrary compact Jordan measurable  sets in $M_D$. Then 
\begin{equation}%\label{Intro6}
 \lim_{t \rightarrow \infty} \frac{ \int_{\Omega_1} d\eta_{it} }{ \int_{\Omega_2} d\eta_{it} }  \, = \, \frac{\text{vol} \, (\Omega_1)}{\text{vol} \, (\Omega_2)},
\end{equation}
where $d\eta_{it} = \abs{E \big(z+\lambda j, \tfrac{1}{2} +it \big) }^2  \, dV$ with $E \big(z+\lambda j,s)$ being the Eisenstein series for $\PSL(2,\mathcal{O}_D) $ and $dV$ is the volume element of $\H^3$.
\end{theo}

\bigskip
For $ \PSL(2, \Z[i]) \backslash \PSL(2,\C) $ the distribution in (\ref{Medida}) is the analogue of the distribution in (\ref{DisZ}). The formula (\ref{Intro5}) of Jakobson's theorem corresponds to the formula in (\ref{QUEPM}).

\bigskip
In section 1 of preliminaries we recall some concepts such as: Wigner matrix, Eisenstein series and your Fourier expansion. In section 2 we review the Mellin transform and we prove the formula in (\ref{PPO2}). In section 3 we will calculate the interior product $ (f, d\epsilon_{it}) $ for when $ f $ is a cusp form and we estimate when $ t \rightarrow \infty $. In section 4 we consider the case where $ f $ is an incomplete Eisenstein series. In the appendix we gather some auxiliary results.

%\bigskip
%\bigskip
\bigskip
\bigskip
\section{Preliminaries}
Let  $K:=\Q( \sqrt{-1})$  be an imaginary quadratic field and $\mathcal{O}= \Z[i]$ its ring of integers. Let $\widetilde{\Gamma} := \PSL(2,\Z[i])$ be the corresponding Bianchi subgroup in $\PSL(2,\C)$ and $\Gamma =\SL(2,\Z[i])$ the  Bianchi subgroup in $\SL(2,\C)$. The quotient ${M}= \widetilde{\Gamma} \backslash \H^3  $ are a tridimensional Bianchi orbifold with only one cusp in $\infty$. Since $\PSL(2,\C)$ can be identified with the orthonormal frame bundle  $\mathcal{F}(\H^3)$ of hyperbolic 3-space then $\mathcal{F}({M})=\widetilde{\Gamma} \backslash \PSL(2,\C)$ can be identified with the orthonormal frame bundle  of ${M}$.

\bigskip
The upper half-space model for hyperbolic space will be used  $ \H^3 $:
\[
  \H^3  \, = \, \{ (z, \lambda) \, ; \, z \in \C , \lambda \in \R^+ \} =
              \{ (x,y, \lambda) \, ; \, x,y \in \R , \lambda > 0 \},
\]
provided with the hyperbolic Riemannian metric: $ ds^2 \, = \, \frac{dx^2+dy^2 + d\lambda^2}{\lambda^2}. $
A point $P \in \H^3$ is denoted as follows:
\[
    P \, = \, (z,\lambda ) = \, z + \lambda j,
\]
where $z=x+iy$, $\lambda>0$, and $j$ is the well know quaternion. The groups $\PSL(2,\C)$ and $\SL(2,\C)$ act on $\H^3$ as follows:
\begin{equation}\label{Pre3}
    \displaystyle  \left(  \begin{matrix}
a & b \\
c & d \\
\end{matrix}  \right)  ( z + \lambda  j ) \, =  \,
    \frac{ (a z  + b)(\bar{c} \bar{z} + \bar{d}) + a \bar{c} \, \lambda^2 }{  \vert cz + d \vert^2 + \vert c \vert^2 \, \lambda^2 } +
         \frac{\lambda}{ \vert cz + d \vert^2 + \vert c \vert^2 \, \lambda^2 }   \, j.
\end{equation}

\begin{rem}The action for $\SL(2,\C)$ is not effective since two matrices that differ by a sign determine the same M\"obius transformation. The group  $\PSL( 2,\C)$ acts effectively on $ \H^3 $  by orientation-preserving isometries of the hyperbolic metric.
\end{rem}

%\bigskip
%El dominio fundamental $\mathcal{D} $ que utilizaremos es el siguiente:
%\begin{equation}\label{DomFu}
%  \mathcal{D} \, := \, \{ (z+\lambda j) \in \H^3 \, ; \, z \in \mathcal{P}, \,  x^2+y^2+\lambda^2 \geq 1  \},
%\end{equation}
%donde $\mathcal{P} \, := \, [-\rfrac{1}{2},\rfrac{1}{2}] \times [0,\rfrac{1}{2}] \subset \C$.

% ========================
\bigskip
If $g \, = \, \left(  \begin{matrix}
a & b \\
c & d \\
\end{matrix}  \right) \, \in \, \SL(2,\C)$ then by the complex Iwasawa decomposition $g$  can be written in a unique way by the following formula:
\[
    \displaystyle \left(  \begin{matrix}
a & b \\
c & d \\
\end{matrix}  \right) = \left(  \begin{matrix}
1 & \frac{ a \bar{c} +  b \bar{d} }{ \vert c \vert^2 + \vert d \vert^2 }\\
0 & 1 \\
\end{matrix}  \right)  \left(  \begin{matrix}
\frac{ 1}{ \sqrt{ \vert c \vert^2 + \vert d \vert^2 } } & 0 \\
0 & \sqrt{ \vert c \vert^2 + \vert d \vert^2 } \\
\end{matrix}  \right)   \left(  \begin{matrix}
\frac{ \bar{d}}{ \sqrt{ \vert c \vert^2 + \vert d \vert^2 } } & \frac{ - \bar{c}}{ \sqrt{ \vert c \vert^2 + \vert d \vert^2 } } \\
\frac{ c}{ \sqrt{ \vert c \vert^2 + \vert d \vert^2 } } & \frac{ d}{ \sqrt{ \vert c \vert^2 + \vert d \vert^2 } } \\
\end{matrix}  \right).
\]

We call  $(z,\lambda, A)$ the \emph{Iwasawa coordinates} of $g$ with

\begin{equation}\label{Pre4}
    \displaystyle  z = \frac{ a \bar{c} +  b \bar{d} }{\vert c \vert^2 + \vert d \vert^2 },  \; \;  \; \; \;
    \lambda = \vert c \vert^2 + \vert d \vert^2, \; \; \; \; \; A = \left(  \begin{matrix}
\frac{ \bar{d}}{ \sqrt{ \lambda } } & \frac{ - \bar{c}}{ \sqrt{ \lambda } } \\
\frac{ c}{ \sqrt{ \lambda } } & \frac{ d}{ \sqrt{ \lambda } } \\
\end{matrix}  \right) \in \SU(2).
\end{equation}
Using Iwasawa coordinates we can write  $g$ as  $  g \, = \, n [ z ] \, a[\tfrac{1}{\lambda}] \, A, $ where
\[
n[z] := \left(  \begin{matrix}
1 & z \\
0 & 1 \\
\end{matrix}  \right) \hspace{4cm}
a[\lambda] := \left(  \begin{matrix}
\sqrt{\lambda} & 0 \\
0 & \frac{1}{\sqrt{\lambda}} \\
\end{matrix}  \right). 
\]

\bigskip
Denote $G:=\SL(2,\C)$ and  $dg$, $dk$ the corresponding Haar measures in $G$ and $K$ respectively that satisfy
$$ \int_{\SU(2)} dk \, = \, 1, \hspace{4cm} dg \, = \, \frac{dx \, dy \, d\lambda \, dk}{\lambda^3}.$$
Also will be denoted by $d V $ the hyperbolic volume measure in $\H^3$, which is
$$ d V \, = \,  \frac{dx \, dy \, d\lambda }{\lambda^3}. $$

\bigskip
Let $l \in \frac{1}{2} \Z$, $l \geq 0$,  two representations of $\SU(2)$ of dimension $2l+1$ will be used,  for that purpose, we  consider vectorial spaces on $\C$  generated by:
$$ V_{2l} \, := \, \{\, z^{l-q} \, ; \, q \in \tfrac{1}{2} \Z, \, q \equiv l \text{ mod } 1, \, q \in [-l,l] \, \},$$
\vspace{-0.3cm}
$$ \mathcal{V}_{2l} \, := \, \{ \, f_m (x,y) \, ; \, m \in \tfrac{1}{2} \Z,  m \equiv l \text{ mod } 1, \, m \in [-l,l] \, \}, \; \; \text{ where } f_m (x,y) := \frac{x^{l+m}y^{l-m}}{\sqrt{(l+m)!(l-m)!}}. $$

\bigskip 
Guleska \cite{Lok04} considers the following representation, $ K \in \SU(2) \longmapsto \sigma_l (K) : V_{2l} \longrightarrow  V_{2l}$ given by
$$ \sigma_l \big( K[\alpha, \beta] \big) \, z^{l-q} \, := \, \big( \alpha z - \overline{\beta} \big)^{l-q} \; \big( \beta z + \overline{\alpha} \big)^{l+q}, \text{ where } 
  K=K[\alpha, \beta] :=  
    \left( \begin{matrix} 
    \alpha & \beta \\
    -\overline{\beta} & \overline{\alpha} \\
    \end{matrix}  \right)$$
    with $\alpha,\beta \in \C$ that satisfy  $\abs{\alpha}^2 + \abs{\beta}^2 =1$. 
We have the next scalar product 
$$ \langle z^{l-q}, z^{l-k}  \rangle \, = \, 
\begin{cases}
0   & \text{if $k \neq q$}\\
(l+q)!(l-q)!    & \text{if $k=q.$}
\end{cases}$$

Using the basis $V_{2l}$ the functions $\Phi_{km}^l : \SU(2) \longrightarrow \C$, that are the matrix coefficients of the reprepresentation $\sigma_l$, are defined by
$$  \Phi_{km}^l (K) \, := \, \frac{1}{(l+k)!(l-k)!} \;  \langle \sigma_l (K) z^{l-m}, z^{l-k}  \rangle, \; \; \forall K \in \SU(2).$$

\bigskip
The basis of functions $\{ \Phi_{km}^l \} $ in $\SU(2)$ is orthogonal, this is
\begin{equation}\label{A10}
\int_{\SU(2)}  \Phi_{km}^l (K) \cdot \overline{ \Phi_{k'm'}^{l'} (K)} \; dk \, = \,
\frac{1}{2l+1} \cdot  \frac{(l+m)!(l-m)!}{(l+k)!(l-k)!} \; \delta_{l,l'} \,  \delta_{k,k'}  \delta_{m,m'}.
\end{equation} 

\bigskip 
Wigner  \cite{Wigner} uses the next  representation,  $ K \in \SU(2) \longmapsto  \textbf{P}_K^l : \mathcal{V}_{2l} \longrightarrow \mathcal{V}_{2l}$ given  by
$$
\textbf{P}_K^l f_m (x,y)   \, := \,
f_m \bigg( K^{-1}
    \left(  \begin{matrix}
x  \\
y  \\
\end{matrix}  \right)
 \bigg). $$

The functions $\mathfrak{U}_{km}^{l} :\SU(2) \longrightarrow \C$ are determined by the next formula 
\begin{equation*}%\label{PAfm10}
     \displaystyle   \textbf{P}_K^l f_m (x,y) \, = \, \sum_{k= -l}^l  \mathfrak{U}_{km}^{l} ( K)   \cdot f_k (x,y).
\end{equation*}

%\bigskip
The relationship beetween the matrix coefficients  in the different basis is given by
\begin{lemm}
Let $l,k,m \in \tfrac{1}{2} \Z$, $l \geq 0$, such that  $l \equiv k \equiv m \text{ mod } 1$, $k,m \in [-l,l]$. Then 
\begin{equation}\label{A11}
\Phi_{km}^l (K) \, = \, \frac{\sqrt{(l+m)!(l-m)!}}{\sqrt{(l+k)!(l-k)!}} \; \cdot  
 \mathfrak{U}_{km}^{l} \big( I K I^{-1} \big), \; \; \forall K \in \SU(2),
\end{equation}
where  $ I := \left(  \begin{matrix}
i & 0 \\
0 & -i \\
\end{matrix}  \right).$
\end{lemm}
\begin{proof}
Direct calculations.
\end{proof}

%\bigskip
%Por el Lema 2.2.1 en \cite{Lok04} las funciones  $\{ \Phi_{km}^l \} $ cumplen que
%\begin{equation}\label{Eigenf}
% \Omega_{\mathfrak{k}} \, \Phi_{km}^l \, = \, - \frac{1}{2} (l^2+l) \,  \Phi_{km}^l, \hspace{3cm} H_2 \, \Phi_{km}^l \, = \, -i m \, \Phi_{km}^l.
%\end{equation}

%\bigskip
%Sea 
%\[
%  L^2 (K; l,m) \, := \, \{ \, f \in L^2(K) \, ; \, 
%   \Omega_{\mathfrak{k}} \, f \, = \, - \frac{1}{2} (l^2+l) \,  f, H_2 \, f \, = \, -i m \, f \, \},
%\]
%el subespacio llamado de tipo $(l,m)$. M\'as a\'un, 
%\[
%  L^2 (K; l,m) \, = \, \displaystyle
%    \bigoplus_{\substack{ k \equiv l \text{ mod } 1 \\ \abs{k} \leq l }} \; \C \, \Phi_{km}^l. 
%\]

\bigskip
We remember the double covering epimorphism  $\Phi : SU(2) \longrightarrow SO(3)$ which is denoted for  $ \Phi ( A ) = \Phi_A$, then
\begin{equation}\label{HomorfismoPhi2}
\Phi_{ AB} = \Phi_{A} \circ \Phi_{B}, \; \;  \forall A,B \in SU(2).
\end{equation}
\vspace{-0.4cm}
\begin{equation}\label{WI13}
\Phi_{\pm I} = I,\quad\quad \text{kernel}\,\Phi=\left\{I,-I \right\}.
\end{equation}
This is of course  the spin cover of $SO(3)$. Explicitly, if $ A = \left(  \begin{matrix}
\alpha & \beta \\
-\overline{\beta} & \overline{\alpha} \\
\end{matrix}  \right) \in SU(2)$, then
\begin{equation}\label{AuxM1}
   \Phi_A \, = \, \left(  \; \begin{matrix}
  \text{Re} \, (\alpha^2 - \beta^2)                & - \text{Im} \, (\alpha^2 + \beta^2)          & 2 \,  \text{Re} \, ( \alpha \beta ) \\
  \text{Im} \, ( \alpha^2 -  \beta^2 )             & \text{Re} \, ( \alpha^2 +  \beta^2 )             &  2 \, \text{Im} \, (  \alpha \beta  ) \\
 - 2 \, \text{Re} \, (  \overline{\alpha} \beta )  & 2 \, \text{Im} \, (  \alpha \overline{\beta }) &   \abs{\alpha}^2 - \abs{\beta}^2  \\
\end{matrix} \;  \right).
\end{equation}

\bigskip
A consequence of the Peter-Weyl theorem is that the functions   $D_{km}^{l} :\SU(2) \longrightarrow \C$  are a basis of $L^2 \big( \SU(2), d k \big)$. As in Lachieze-Rey \cite{Lachieze}, we choose the basis of $L^2 \big( \SU(2), d k \big)$ next: 
$$ \{\, \mathcal{T}^{\, l}_{km} \, ; \, l \in  \Z, l \geq 0 \, ; \,  k,m \in \tfrac{1}{2} \Z, k \equiv m \equiv \tfrac{l}{2} \text{ mod } 1, \text{ and } k,m \in [-\tfrac{l}{2},\tfrac{l}{2}] \, \},$$
where 
\begin{equation}\label{T1}
 \mathcal{T}^{\, l}_{km} (A) \, := \, \sqrt{\frac{l+1}{2 \pi^2}} \cdot D_{mk}^{\rfrac{l}{2}} \big(\Phi (A) \big),   \; \; \forall A \in \SU(2).
\end{equation}

\bigskip
Let $R \in SO(3)$, we can write $R$ in terms of its Euler angles:
\[
      R \, = \, ROT(\theta, \chi, \phi) \; : = \;
      \left(  \begin{matrix}
 \cos \theta   & \sin \theta  & 0  \\
-\sin \theta   & \cos \theta  & 0   \\
0 & 0 &   1   \\
\end{matrix}  \right)
 \left(  \begin{matrix}
 \cos \chi   & 0   & - \sin \chi  \\
0    & 1  & 0   \\
 \sin \chi  & 0 &   \cos \chi   \\
\end{matrix}  \right)
  \left(  \begin{matrix}
 \cos \phi   & \sin \phi  & 0  \\
-\sin \phi   & \cos \phi  & 0   \\
0 & 0 &   1   \\
\end{matrix}  \right),
\]
with  $\theta \in [0,2\pi), \chi \in [0,\pi], \phi \in [-2\pi,2\pi)$. In addition we have,
\begin{equation}
\label{WignerIg}
     \displaystyle  D_{km}^l \big( R  \big) \, = \, D_{km}^l \big( ROT(\theta, \chi, \phi)  \big) \, = \, e^{i k \theta } \, d_{km}^l (\chi) \,  e^{i m \phi},
\end{equation}
where $d_{km}^l (\chi)$ is \emph{Wigner small $d$-matrix}. We refer to the book  of Wigner  \cite{Wigner} in order to see how the coefficients $D_{km}^l \big( R  \big)$ are deduced. 

\bigskip
Later we use the following property of Wigner matrices. Let $R,T \in \SO(3)$, $l,k,m \in \tfrac{1}{2}\Z$, $l \geq 0$ and such that $k,m \in [-l,l]$. Then
\begin{equation}\label{WI10}
    D^l_{km} (R \circ T) \, = \, \sum_{\substack{ a=-l \\ a \in \tfrac{1}{2} \Z \\ a \equiv l \text{ mod } 1 }}^{l} D^l_{ka} (R ) \cdot D^l_{am} ( T).
\end{equation}

This is a consequence of the fact that for fixed $l$, the entries $ D^l_{bm}$ are the coefficients of a representation. 

\bigskip
In order to see how the basis changes  $\{ \mathcal{T}^{\, l}_{km} \}$  under rotations, we remember the action of $\SU(2)$ in the smooth functions defined in $\SU(2)$ $$ A \in \SU(2) \longmapsto  R_A : C^\infty \big(\SU(2)\big) \longrightarrow C^\infty \big(\SU(2)\big)$$  by
\begin{equation}\label{RotaenK}
R_A : f \longmapsto R_A f \text{ with } R_A f (K) \, := \, f (A^{-1} \, K), \; \; \forall f \in C^\infty \big(\SU(2)\big), \; \; \forall K \in \SU(2).
\end{equation}

\bigskip
Then, 
\begin{equation}\label{B2}
R_{A^{-1}} \,   \mathcal{T}^{\, l}_{km} (K) \, = \, 
   \sqrt{\frac{l+1}{2 \pi^2}} \cdot D_{mk}^{\rfrac{l}{2}} \big(\Phi (AK) \big),   \; \; \forall A,K \in \SU(2).
\end{equation}  %  \mathcal{T}^{\, l}_{km} (AK) \; = \;

\bigskip
Replacing (\ref{T1}) and (\ref{B2}) in (\ref{WI10}) it is concluded that the basis  $\{ \mathcal{T}^{\, l}_{km} \}$ changes under rotations according to the next formula:
\begin{equation}\label{L19.1}
   R_{A^{-1}}  \mathcal{T}^{\, l}_{km}  \, = \,
    \sum_{\substack{ a=-\rfrac{l}{2} \\ a \in \tfrac{1}{2} \Z \\ a \equiv \rfrac{l}{2} \text{ mod } 1 }}^{\rfrac{l}{2}} {D}^{\rfrac{l}{2}}_{ma} \big ( \Phi(A)  \big)  \cdot  \mathcal{T}^{\, l}_{ka}, \; \, \forall A \in \SU(2). 
\end{equation}

\bigskip
Now we remember other  properties of Wigner matrix that we further use
\begin{equation}\label{A9}
 \mathfrak{U}_{km}^{l} \big(  K  \big) \, = \, 
 D_{km}^l \big( \Phi (  K )  \big), \; \; \forall K \in \SU(2).
\end{equation}
\begin{equation}\label{A7}
D_{km}^l \big( \Phi ( I K I^{-1} )  \big) \, = \, (-1)^{m-k} \cdot
D_{km}^l \big( \Phi (  K )  \big), \; \; \forall K \in \SU(2).
\end{equation}
\begin{equation}\label{A8}
D_{km}^l \big( \Phi (  K )^{-1}  \big) \, = \, (-1)^{m-k} \cdot  
D_{-m,-k}^l \big( \Phi (  K )  \big) \, = \, \overline{  D_{mk}^l \big( \Phi (  K )  \big)}, \; \; \forall K \in \SU(2).
\end{equation}

%Using the properties $d_{km}^l (\chi) = d_{-m,-k}^l (\chi)$, $d_{km}^l (\chi) = (-1)^{m-k} \, d_{mk}^l (\chi)$ we have that
%On the other hand, from (\ref{A4}) and (\ref{A9}) y (\ref{A7})
%\begin{equation}\label{A11}
%\Phi_{km}^l (K) \, = \,  \frac{\sqrt{(l+m)!(l-m)!}}{\sqrt{(l+k)!(l-k)!}} \; (-1)^{m-k} \cdot  D_{km}^l \big( \Phi (  K )  \big), \; \; \forall K \in \SU(2).
%\end{equation}

\bigskip
Let $l,k,m \in \tfrac{1}{2} \Z$, $l \geq 0$, $l \equiv k \equiv m \text{ mod } 1$ and  such that $k,m \in [-l,l]$. We denote by ${\scriptstyle\Gamma'_\infty}\backslash\Gamma$ the set of right cosets of $\Gamma'_{\infty}$ in $\Gamma$. We remember that $\Gamma'_\infty$ is the maximal unipotent subgroup, i.e.,
\[
     \Gamma'_\infty \, = \, \bigg\{
    \left(
\begin{matrix}
1 & z \\
0 & 1 \\
\end{matrix}  \right) \;  ; \; z \in Z[i]  \bigg\}.
\]

If $s \in \C$ is such that $\text{Re}(s) > 1$,  $f_{km}^l(\cdot, s) : \SL(2,\C) \longrightarrow \C $ is defined by the formula:
\begin{equation}\label{Deffff}
   f_{km}^l(g,s) \, := \, \overline{ D_{km}^l \big( \Phi_{ T( g )^{-1}} \big) } \; \text{Im} \,  g (j)^{1+s}.
\end{equation}
The Eisenstein series ${E}_{km}^l(\cdot,s): \SL(2,\C) \longrightarrow \C$ associated to $\Gamma$ at the cusp $\infty$ are defined by the following  formula:
\begin{equation}\label{ES}
  \displaystyle {E}_{km}^l (g,s) \, := \,
       \sum_{\sigma\,\in\,_{\scriptstyle\Gamma'_\infty}\backslash\Gamma} f_{km}^l (\sigma g,s ).
\end{equation}

\begin{rem}
In \cite{Otto} the previous series are denoted as  $\widehat{E}_{km}^l (g,s)$. Furthermore, we omit the factor $ \frac{1}{[\Gamma_{\infty} : \Gamma'_{\infty} ]}$ that appears in this work.
\end{rem}

The series $E_{km}^l (g,s) $ are well defined, they are convergent to $\Re(s) >1$ and are invariants under $\Gamma$. So we define the series $
 E_{km}^l(\cdot,s): \mathbb{H}^3 \longrightarrow \C$ associated to $\Gamma$ at the cusp  $\infty$  by the formula:
\begin{equation}\label{def0004}
   E_{km}^l (z+ \lambda j,s) \, := \, E_{km}^l \big(  n[z] a[\lambda],s ). 
\end{equation}

\bigskip
The identity (\ref{def0004}) implies that the series $E_{km}^l (z+ \lambda j,s)$ are well defined and are convergent in the half-plane  $\text{Re}(s) > 1$. Although  $E_{km}^l (z+ \lambda j,s)$  do not define in general (except for $l=k=m=0$) Eisenstein series in $\mathbb{H}^3$ (since they are not invariant under $\Gamma$), they do admit a Fourier expansion. 

\bigskip
In \cite{Langlands}  Langlands developed the general theory of Eisenstein series, proving that they satisfy a functional equation and the analytic or meromorphic continuation of ${E}_{km}^l (g,s)$ with respect $s$ to all the complex plane. 
%This theory is enhanced in \cite{MW},  chapter IV, where a proof of meromorphic continuation of Eisenstein series is given and attributed to Herv\'e Jacquet. See also \cite{Garrett} on the web page of Paul Garrett.
Also, since the Bianchi groups are special, the analytic or meromorphic continuation can be obtained without using the Langlands work, see \cite{BruMoto}, page 44, for a proof in our particular case $\Gamma=\PSL(2, \Z[i])$.

\bigskip
In the following sections, the Fourier expansion of the series $E_{km}^l (z+ \lambda j,s)$ will be used, before this, it is necessary to recall some important facts. Let $ n \in  4\Z$, we define a Hecke character $\chi_n$ for  $Q(\sqrt{-1})$ as follows:
$$ \chi_n \big((c)\big) \, := \, \bigg( \frac{c}{\abs{c}} \bigg)^n, \; \; 0 \neq c \in Z[i],$$
and where  $(c)$ denotes the ideal generated by  $c$. The Hecke L function $ L(s, \chi_{n} ) $ are defined by:
\[
     L(s, \chi_{n} )  \, = \,  \sum_{   \mathfrak{a} }   \frac{ \chi_{n} (\mathfrak{a})  }{ \big[ \textit{N} (\mathfrak{a}) \big]^s  } =
       \prod_{ \substack{ p \in Z[i]  \\  p  \text{ primo}  } } \frac{1}{ 1 - \chi_{n} \big((p)\big) \cdot \abs{p}^{-2s}  },
\]
for $s \in \C$ such that $\text{Re}(s)>1$ and the sum runs over all  integral ideals 
 $\mathfrak{a}$ not zero of $Z[i]$. We denoted by $\zeta_K(s)$ to the Dedekind zeta function associated to $Q\sqrt{-1}$. In particular, $L(s, \chi_{0} ) =
\zeta_K(s)$.  

\bigskip
The Fourier expansion of series $ E_{k,m}^l (z+\lambda j,s)$ (see Guleska \cite{Lok04} or Watt \cite{Watt}, page 17) is given by:
$$   E_{k,m}^l (z+\lambda j,s) \, = \, \delta_{k,m} \, [\Gamma_{\infty} : \Gamma'_{\infty} ] \, B_{km}^l \, \lambda^{1+s} \, + \, (-1)^{m+\abs{m} } \pi  \, \frac{\Gamma(1+l-s) \, \Gamma(\abs{m}+s)}{\Gamma(1+l+s) \, \Gamma(1+\abs{m}-s)} \,
\frac{L\big( s,\chi_{2m} \big)}{L( s+1, \chi_{2m})} \,  \delta_{-k,m} \, B_{km}^l \, \lambda^{1-s} $$
$$ \, + \,  (-1)^{l+m} \,  i^{-k-m} \, (2 \pi)^s  \,    B_{km}^l \,    \Large \sum_{ 0 \neq w \in \Lambda'} 
\mathcal{D}_{m} (-w; s) \cdot \abs{ w }^{s-1} \cdot e^{-4 \pi i \, \text{Re}(w z )}    \Big( \frac{w}{\abs{w}}  \Big)^{-k-m} \hspace{3cm} $$
$$ \cdot  \Large \sum_{ u=0 }^{ l- \frac{1}{2}( \abs{k+m}+\abs{k-m})} (-1)^u \; 
 \xi_{-m}^l (-k,u)  \, \frac{\big(2 \pi \abs{w} \lambda \big)^{1+l-u}}{\Gamma(1+l+s-u)} 
\cdot  K_{s+l-\abs{k+m}-u} \big( 4 \pi \abs{w} \lambda \big),  $$
with $\Lambda \, := \, Z[i]$ and
\[
   \Lambda' \, := \, \{ z \in \C \; ; \; \Re(z \alpha) \in \tfrac{1}{2} \Z \; \; \forall \alpha \in \Lambda \} \, = \, \tfrac{1}{2} \, Z[i],
\]
\begin{equation}\label{D_k}
\mathcal{D}_k (w; s) \, := \,  
\sum_{ 
    \left( \begin{matrix} * & * \\ c & d \\ \end{matrix}  \right) \in  \mathcal{R}}
 \frac{1}{\abs{c}^{2+2s}} \bigg( \frac{c}{\abs{c}}  \bigg)^{2k} \cdot 
e^{4 \pi i \, \text{Re}(w \, \frac{d}{c} )}, \; \, \forall k \in \Z
\end{equation}
where
\[
  \mathcal{R} \, := \,  \left \{
    \left( \begin{matrix} * & * \\ c & d \\ \end{matrix}  \right) \in \Gamma   \; ; \;
     c \neq 0, \;  d  \text{ mod} \, c
   \right \},
\]
and 
\[
 \xi_k^l (v,u) \, := \, \frac{u! (2l-u)!}{(l+k)!(l-k)!} \, \binom{l- \frac{1}{2}( \abs{v+k}+\abs{v-k})}{u} 
\binom{l- \frac{1}{2}( \abs{v+k}-\abs{v-k})}{u}, \hspace{1cm}
B_{km}^l \, := \; \frac{\sqrt{(l+m)!(l-m)!}}{\sqrt{(l+k)!(l-k)!}}.
\]

\bigskip
We remember that for  $\nu \in \C$  a twisted divisor function is given by 
\begin{equation}\label{sigma_nu}
\sigma_{\nu} (w,p) \, := \, \frac{1}{4} \, \sum_{d \mid w} \chi_{4p} \big((d)\big) \cdot \abs{d}^{2 \nu}, \; \; \forall w \in Z[i], \; \; \forall p \in \Z.
\end{equation}
%\begin{rem} Usaremos algunos hechos de  Bruggeman and Motohashi\cite{BruMoto}, trabajo en el que la notaci\'on para  $\sigma_{\nu} (w,p)$ es diferente.
%\end{rem}

For any $k \in 2 \Z$ %(see formula 2.18 en \cite{BruMoto})
\begin{equation}\label{IE20}
\mathcal{D}_p (w; s) \, = \, \frac{16}{L(1+s,\chi_{2p}) }
\cdot \begin{cases}
\hfill \sigma_{-s} (2w,  \tfrac{p}{2} )& \text{if $ w \neq 0, \text{ Re}(s) > 0$} \\
  \tfrac{1}{4} \;   L (s, \chi_{2p}) & \text{if $w = 0, \text{ Re}(s)  < 1.$ }
 \end{cases}
\end{equation}

%\bigskip
%Por otro lado, tenemos que
%$$ \mathcal{D}_0 ( -w ;it) \, = \, \frac{4^2}{\zeta_K (1+it)} \cdot \sigma_{ -it} (-2w,0) \, = \, \frac{4^2}{\zeta_K (1+it)} \cdot \sigma_{-it} (2w,0).$$

\section{Mellin transform}

The aim in this section is to prove the formula in (\ref{PPO2}). Before we require some  notation. Let $(q,\eta) \in \H^3$, $F_{(q,\eta)} := \{ \, n[q] a[\eta] K \, ; \, K \in \SU(2) \, \}$, the ``fiber'' of $(q,\eta)$ in $G$, and consider the natural diffeomorphism  $\Psi_{(q,\eta)} : F_{(q,\eta)}  \longrightarrow \SU(2)$ defined by $  \Psi_{(q,\eta)} \big(\,n[q] a[\eta] K \,\big) = K$. 

\bigskip
For  $\gamma \in G$, ${M}_\gamma(q+\eta j) : \SU(2)   \longrightarrow \SU(2)$ is the function defined as follows:
$$ {M}_\gamma(q+\eta j) \, K \, := \, T \big( \gamma \, n[q] a[\eta] K \big) \, = \,
T \big( \gamma \, n[q] a[\eta] \big) K, \; \; \forall K \in \SU(2). $$

\bigskip
Let $f \in C^\infty \big( \Gamma \backslash G \big)$, equivalently, $f:G  \longrightarrow \C$ (use the same letter) is a smooth function and  $\Gamma$ invariant. This is, $f(g) = f( \gamma g), \; \; \forall \gamma \in \Gamma
, \forall g \in G \; \; \, \iff$
\begin{equation}\label{L3}
  f \big( n[z] a[\lambda] K \big) \, = \,
  f \big( n[z'] a[\lambda'] {M}_\gamma (z+ \lambda j) \, K \big), \text{ where } z'+\lambda' j= \gamma(z+\lambda j), \; \; \forall z+\lambda j \in \H^3, \; \; \forall \gamma \in \Gamma, \; \; \forall K \in \SU(2).
\end{equation}

Given that $G \cong \H^3 \times \SU(2)$ also use the notation  
$$ f \big( z+\lambda j,  K \big) \, = \, f \big( \gamma(z+\lambda j), {M}_\gamma(z+ \lambda j) \, K \big),  \; \; \forall z+\lambda j \in \H^3, \; \; \forall \gamma \in \Gamma, \; \; \forall K \in \SU(2).$$

% {\substack{k,m=-l \\ k \equiv m \equiv l \text{ mod }1 } 
%  {\substack{ l \in  \Z \\ l \geq 0 }} 
\bigskip
On the other hand, restricting  $f$ to the fibers we have that  $f \circ \Psi_{(z,\lambda)}^{-1}: \SU(2) \underset{\text{difeo.}}{\cong}  \S^3  \longrightarrow \C$  for every  $(z,\lambda) \in \mathbb{H}^3$, for which, suposing that  $f \circ \Psi_{(z,\lambda)}^{-1}$ admits a Laplace series, we have the following expansion:
\begin{equation}\label{L7}
  f \circ \Psi_{(z,\lambda)}^{-1} (K) \, = \,
   \sum_{\substack{ l \in  \Z \\ l \geq 0 }}  \; 
   \sum_{\substack{ k,m=-\rfrac{l}{2} \\ k,m \in \tfrac{1}{2} \Z \\ k \equiv m \equiv \rfrac{l}{2} \text{ mod } 1 }}^{\rfrac{l}{2}}   \widehat{f}^{\, l}_{k,m} (z+\lambda j) \cdot   \mathcal{T}^{\, l}_{km} (K) %, \; \; l \text{ par},
\end{equation}
where
\begin{equation}\label{L8}
  \widehat{f}_{k,m}^{\, l}  (z+\lambda j)   \, = \,
    2\pi^2   \int_{\SU(2)} f \circ \Psi_{(z,\lambda)}^{-1} (K)
       \cdot \overline{  \mathcal{T}^{ \,l}_{km} (K) } \; dk.
\end{equation}

\bigskip
Let $\gamma \in \Gamma$, for (\ref{L8})
\begin{equation}\label{L9}
  \widehat{f}_{k,m}^{\, l}  \big( \gamma(z+\lambda j) \big)   \, = \,
     2\pi^2  \int_{\SU(2)} f \circ \Psi_{\gamma(z,\lambda)}^{-1} (K)
       \cdot \overline{  \mathcal{T}^{\,l}_{km} (K) } \; dk,
\end{equation}
however for  (\ref{L3})
\begin{equation}\label{L10}
 f \circ \Psi_{\gamma(z,\lambda)}^{-1}   \, = \,
       f \circ \Psi_{(z,\lambda)}^{-1} \circ  {M}_\gamma(z+\lambda j)^{-1}.
\end{equation}
Replacing (\ref{L10}) in (\ref{L9})
\begin{equation}\label{L11}
  \widehat{f}_{k,m}^{\, l}  \big( \gamma(z+\lambda j) \big)   \, = \,
      2\pi^2  \int_{\SU(2)} f \circ \Psi_{(z,\lambda)}^{-1} 
      \big[ {M}_\gamma(z+\lambda j)^{-1}  \, K \big]
       \cdot \overline{  \mathcal{T}^{\,l}_{km} (K) } \; dk.
\end{equation}
We make the variable change in $\SU(2)$ given by 
\begin{equation}\label{L12}
  {M}_\gamma(z+\lambda j)^{-1}  \, K \, = \, K'.
\end{equation}

\bigskip
However if $\gamma = 
 \left( \begin{matrix}
a  &  b \\
c  &  d  
\end{matrix}  \right) \in \SL(2,\C)$ is direct to verify that
$$  {M}_\gamma(z+ \lambda j)  \, K' \, = \, \frac{1}{v} \,  
 \left( \begin{matrix}
\overline{c}\overline{z}+\overline{d}    &  - \lambda \,\overline{c} \\
 \lambda \,  c  &  cz+d 
\end{matrix}  \right)  K', \; \;  \forall K' \in \SU(2),$$
where $v:= \sqrt{\lambda^2 \abs{c}^2 + \abs{cz+d}^2}.$ 
%Also
%$$ T(\gamma g) \, K' \, = \,  \frac{1}{v} \, \left( \begin{matrix}
%\overline{c}\overline{z}+\overline{d} &  - \lambda \overline{c}  \\
%\lambda c &  cz+d \\
%\end{matrix}  \right)  K',  \; \;  \forall K' \in \SU(2) $$
Therefore  
\begin{equation}\label{MG}
 {M}_\gamma(z+ \lambda j) \, K' \, = \, T(\gamma g)  \; K', \; \;  \forall K' \in \SU(2),
\end{equation}
with $g \,  := \, n[z] a[\lambda]$.

\bigskip
We revisit the integral of $ \widehat{f}_{k,m}^{\, l}  \big( \gamma(z+ \lambda j) \big) $. Replacing  (\ref{L12}) in the equation (\ref{L11}) 
\begin{equation}\label{L17}
 \widehat{f}_{k,m}^{\, l}  \big( \gamma(z+\lambda j) \big)   \, = \,
      2\pi^2  \int_{\SU(2)} f \circ \Psi_{(z,\lambda)}^{-1} (K')
            \cdot \overline{  \mathcal{T}^{\, l}_{km} \big(  {M}_\gamma(z+\lambda j) K'  \big)  } \; dk'.
\end{equation}

%\bigskip
Using the identity  (\ref{MG}) and (\ref{RotaenK}) it is known that
\begin{equation}\label{L18}
  \mathcal{T}^{\, l}_{km} \big(  {M}_\gamma(z+\lambda j) K'  \big) \, = \,
 R_{T(\gamma g )^{-1}}  \mathcal{T}^{\, l}_{km} ( K').
\end{equation}
%where $g = n[z] a[\lambda]$.

\bigskip
As see in  (\ref{L19.1}), the basis of the   $\{ \mathcal{T}^{\, l}_{km} \}$ changes under rotations according to the following formula: 
\begin{equation}\label{L19}
  R_{T(\gamma g )^{-1}}  \mathcal{T}^{\, l}_{km}   \, = \,
      \sum_{\substack{ u=-\rfrac{l}{2} \\ u \in \tfrac{1}{2} \Z \\ u \equiv \rfrac{l}{2} \text{ mod } 1 }}^{\rfrac{l}{2}} {D}^{\rfrac{l}{2}}_{mu} \big ( \Phi_{T(\gamma g )}  \big)  \cdot  \mathcal{T}^{\, l}_{ku}. 
\end{equation}

Summarizing, of (\ref{L18}) and (\ref{L19}),
\begin{equation}\label{L20}
 \mathcal{T}^{\, l}_{km} \big(  {M}_\gamma(z+\lambda j) K'  \big) \, = \,
      \sum_{\substack{ u=-\rfrac{l}{2} \\ u \in \tfrac{1}{2} \Z \\ u \equiv \rfrac{l}{2} \text{ mod } 1 }}^{\rfrac{l}{2}} {D}^{\rfrac{l}{2}}_{mu} \big ( \Phi_{T(\gamma g )}  \big)  \cdot  \mathcal{T}^{\, l}_{ku} (K'). 
\end{equation}

\bigskip
Replacing  (\ref{L20}), (\ref{L8}) and (\ref{A8}) in (\ref{L17}) it is obtained
\begin{lemm}\label{LP}
Let $z+ \lambda j \in \mathbb{H}^3$, $f \in C_0^\infty \big(\Gamma \backslash G \big)$, $\gamma \in \Gamma$. For $l \in  \Z$, $ l \geq 0$,  $k,m  \in \tfrac{1}{2} \Z$,   $k,m \in [-\tfrac{l}{2},\tfrac{l}{2}]$  and  $ k \equiv  m \equiv \tfrac{l}{2} \text{ mod } 1$ it is true that 
\[
   \widehat{f}_{k,m}^{\, l}  \big( \gamma(z+\lambda j) \big)   \, = \,  \sum_{\substack{ u=-\rfrac{l}{2} \\ u \in \tfrac{1}{2} \Z \\ u \equiv \rfrac{l}{2} \text{ mod } 1 }}^{\rfrac{l}{2}} (-1)^{u-m} \;
      \overline{{D}^{\rfrac{l}{2}}_{-u,-m} \big ( \Phi_{ T(\gamma g )^{-1}} \big)} \cdot     \widehat{f}_{k,u}^{\, l}  ( z+\lambda j).
\]
\end{lemm}

\bigskip
Some known facts about the Mellin transform are recalled. For every  $\lambda \in \R$ with  $\lambda > 0$ consider the surface:
\[
    \mathcal{S}_\lambda \, := \, \{ (z+\lambda j,I)   \; ; \;  z \in \C  \}  \subset  G.
\]
$\mathcal{S}_\lambda$ passes to the quotient  $ \Gamma \backslash G  $ as a closed orbifold denoted by $\widetilde{T}_\lambda $.

\begin{Definition-non}%\label{nulambda}
Let $\lambda \in \R$ such that  $\lambda > 0$, the probability measure  $\nu (\lambda)$  located at   $\widetilde{T}_\lambda \subset \Gamma \backslash G $ are defined as follows:
 \[
     \nu (\lambda ) (f) \, : = \,   \int_{\R^2 / Z[i]}  f ( z+ \lambda j, I ) \, dxdy,
     \; \; \; \forall f \in C^{\infty}_{0} ( \Gamma \backslash G  ).
 \]
\end{Definition-non}

%\bigskip
\begin{defi}\label{Mellinf}
If $s \in \C$ such that  $\text{Re}(s) >1$, the Mellin transform  $\mathfrak{M} (f,s)$ of $f \in  C^{\infty}_{0} ( \Gamma \backslash G  )$ is defined by the following equality:
\[
\mathfrak{M}(f,s)    \, := \,  \int_0^\infty \nu (\lambda)(f) \, \lambda^{s-2}  d\lambda  \, = \,   \int_0^\infty \int_{\R^2 / Z[i] } f ( z + \lambda j, I ) \, \lambda^{s+1}  \; \frac{dxdyd\lambda}{\lambda^3}, \; \; z=x+iy.
\]
\end{defi}

\bigskip
Let  $  S \, = \, \{ \, (z,\lambda) \in \mathbb{H}^3 ; z \in [-\tfrac{1}{2}, \tfrac{1}{2} ]^2 \,  \}$, the next identity is true
\begin{equation}\label{LS3}
      \displaystyle S= \bigcup_{\sigma\,\in\,_{\scriptstyle\Gamma'_\infty}\backslash\Gamma} \sigma(F),
\end{equation}
where $F \, := \, \{ (z,\lambda) \in \H^3 ; z \in [-\tfrac{1}{2}, \tfrac{1}{2} ]^2 \,, \, \abs{z}^2 + \lambda^2 \geq 1  \}$.

\begin{prop}\label{PPr} Let  $f \in C^\infty_0 \big( \Gamma \backslash G  \big)$, $s \in \C$ such that  $\text{Re}(s)>1$. Then
\[
 \mathfrak{M}(f,s)   \, = \,  \sum_{\substack{ l \in \Z \\ l \geq 0 }}  \;  \sum_{\substack{ k,u=-\rfrac{l}{2} \\ k \equiv u \equiv \rfrac{l}{2} \text{ mod } 1 }}^{\rfrac{l}{2}}  \, (-1)^{u-k} \cdot \mathcal{T}^{\, l}_{kk} (I)     \int_{F}    \widehat{f}_{k,u}^{\, l} ( z+\lambda j) \cdot  {E}_{-u,-k}^{\rfrac{l}{2}} (z +\lambda j,s)  \; \frac{dx dy d\lambda}{\lambda^3}.
\]
\end{prop}
\begin{proof}
 It is first observed that for  (\ref{L7})
\begin{equation}\label{L19.5}
f \big( z+\lambda j, I \big) \, = \,
  \sum_{\substack{ l \in \Z \\ l \geq 0 }}  \;   \sum_{\substack{ k=-\rfrac{l}{2} \\ k \in \tfrac{1}{2} \Z  \\ k \equiv  \rfrac{l}{2} \text{ mod } 1 }}^{\rfrac{l}{2}}  \widehat{f}^{\, l}_{k,k} (z+\lambda j) \cdot \mathcal{T}^{\, l}_{kk} (I).
\end{equation}

There is the following chain of equalities:   
\begin{align}
 \mathfrak{M}(f,s)  & \, = \, \int_0^\infty \int_{[-\rfrac{1}{2}, \rfrac{1}{2} ]^2}   f(z +\lambda j, I) \; \text{Im} \, (z+\lambda j)^{1+s} \;  \frac{dx dy d\lambda}{\lambda^3}      &&  \text {definition  (\ref{Mellinf})}  \nonumber \\
             & \, = \,  \sum_{\substack{ l \in \Z \\ l \geq 0 }}  \;  \sum_{\substack{ k=-\rfrac{l}{2} \\ k \in \tfrac{1}{2} \Z  \\ k \equiv  \rfrac{l}{2} \text{ mod } 1 }}^{\rfrac{l}{2}}   \,  \mathcal{T}^{\, l}_{kk} (I)   \int_{S}   \widehat{f}_{k,k}^{\, l} (z+\lambda j)   \;    \text{Im} \, (z+ \lambda j)^{1+s} \; \frac{dx dy d\lambda}{\lambda^3}      && \text {by  (\ref{L19.5})}  \nonumber \\
             & \, = \,  \sum_{\substack{ l \in \Z \\ l \geq 0 }}  \;  \sum_{\substack{ k=-\rfrac{l}{2} \\ k \in \tfrac{1}{2} \Z  \\ k \equiv  \rfrac{l}{2} \text{ mod } 1 }}^{\rfrac{l}{2}}   \,  \mathcal{T}^{\, l}_{kk} (I) \sum_{\sigma\,\in\,_{\scriptstyle\Gamma'_\infty}\backslash\Gamma}
              \int_{\sigma(F)}   \widehat{f}_{k,k}^{\, l} (z+\lambda j)   \;      \text{Im} \, (z+\lambda j)^{1+s} \; \cfrac{dx dy d\lambda}{\lambda^3}.      && \text {by  (\ref{LS3})}  \label{LS4}
\end{align}

\bigskip
Now we take care of the integral obtained in the equation (\ref{LS4}), for that we consider the next variable change:
\begin{equation*}%\label{LS6}
  z+\lambda j = \sigma (z'+\lambda' j), \; z'=x'+iy.
\end{equation*}
As  $\sigma: F \longrightarrow \sigma(F)$ is an isometry, it is followed that:
\begin{align}
 \int_{\sigma(F)}   \widehat{f}_{k,k}^{\, l} (z+\lambda j)   \; \text{Im} \, (z+\lambda j)^{1+s} \; \frac{dx dy d\lambda}{\lambda^3}  & \, = \,
      \int_{F}  \widehat{f}_{k,k}^{\, l} \big( \sigma (z'+\lambda' j) \big)  \; \text{Im} \, \sigma (z'+\lambda' j)^{1+s} \; \frac{dx' dy' d\lambda'}{\lambda'^3}.
                 &&  \nonumber 
\end{align}
However for the lemma (\ref{LP}) 
\begin{equation}\label{LS8}
   \int_{\sigma(F)}  \widehat{f}_{k,k}^{\, l}(z+\lambda j)  \; \text{Im} \, (z+\lambda j)^{1+s} \; \frac{dx dy d\lambda}{\lambda^3}    \, = \,
  \sum_{\substack{ u=-\rfrac{l}{2} \\ u \in \tfrac{1}{2} \Z \\ u \equiv \rfrac{l}{2} \text{ mod } 1 }}^{\rfrac{l}{2}} (-1)^{u-k} \int_{F}  
        \overline{{D}^{\rfrac{l}{2}}_{-u,-k} \big ( \Phi_{ T(\sigma g )^{-1}} \big)} \cdot   \widehat{f}_{k,u}^{\, l}  ( z+\lambda j) \; \text{Im} \, \sigma (z+\lambda j)^{1+s}   \; \frac{dx dy d\lambda}{\lambda^3}.
\end{equation}

\bigskip
Replacing the identity (\ref{LS8}) in the equation (\ref{LS4}) it is obtained the first identity of the following:
%\begin{align}
%\mathfrak{M}(f,s)   & \, = \, \sum_{\substack{ l \in \Z \\ l \geq 0 }}  \sum_{\substack{ k=-\rfrac{l}{2} \\ k \in \tfrac{1}{2} \Z \\ k \equiv  \rfrac{l}{2} \text{ mod } 1 }}^{\rfrac{l}{2}} \, \mathcal{T}^{\, l}_{kk} (I)
%    \sum_{\sigma\,\in\,_{\scriptstyle\Gamma'_\infty}\backslash\Gamma}  \sum_{\substack{ u=-\rfrac{l}{2} \\ u \in \tfrac{1}{2} \Z \\ u \equiv \rfrac{l}{2} \text{ mod } 1 }}^{\rfrac{l}{2}} (-1)^{u-k} \int_{F}   
 %     \overline{{D}^{\rfrac{l}{2}}_{-u,-k} \big ( \Phi_{ T(\sigma g )^{-1}} \big)} \cdot     \widehat{f}_{k,u}^{\, l}  ( z+\lambda j)  \,
  %     \text{Im} \, \sigma (z+\lambda j)^{1+s} \, \frac{dx dy d\lambda}{\lambda^3}
   %              &&  \nonumber \\
\begin{align}
\mathfrak{M}(f,s)   & \, = \,  \sum_{\substack{ l \in \Z \\ l \geq 0 }}   \; \sum_{\substack{ k,u=-\rfrac{l}{2} \\ k,u \in \tfrac{1}{2} \Z \\ k \equiv u \equiv  \rfrac{l}{2} \text{ mod } 1 }}^{\rfrac{l}{2}}     \, (-1)^{u-k} \cdot \mathcal{T}^{\, l}_{kk} (I)
                \int_{F}   \widehat{f}_{k,u}^{\, l} (z+\lambda j)
                  \Bigg[ \sum_{\sigma\,\in\,_{\scriptstyle\Gamma'_\infty}\backslash\Gamma} \overline{{D}^{\rfrac{l}{2}}_{-u,-k} \big ( \Phi_{ T(\sigma  g )^{-1}} \big)} \;        \text{Im} \, \sigma (z+\lambda j)^{1+s}  \Bigg]
                 \frac{dx dy d\lambda}{\lambda^3}        &&  \nonumber \\
      & \, = \, \sum_{\substack{ l \in \Z \\ l \geq 0 }}  \;    \sum_{\substack{ k,u=-\rfrac{l}{2} \\ k,u \in \tfrac{1}{2} \Z \\ k \equiv u \equiv \rfrac{l}{2} \text{ mod } 1 }}^{\rfrac{l}{2}}  \, (-1)^{u-k} \cdot \mathcal{T}^{\, l}_{kk} (I)
                 \int_{F}    \widehat{f}_{k,u}^{\, l} ( z+\lambda j)  \cdot  {E}_{-u,-k}^{\rfrac{l}{2}} (z +\lambda j,s)       \; \frac{dx dy d\lambda}{\lambda^3}.
 && \nonumber
\end{align}
\end{proof}

\bigskip
\bigskip
\bigskip
\bigskip
\bigskip
\section{Cusp forms}
We remember that the aim in this section is to encounter a formula for the inner product  $\Big(  F_{p,q}^l  ,  d\epsilon_{it}  \Big)$, we also see that 
 \[
\lim_{t \rightarrow \infty}  \Big(  F_{p,q}^l  ,  d\epsilon_{it}  \Big) \, = \, 0.
\]

\bigskip
The following version of cusp forms is taken like Bruggeman and Motohashi \cite{BruMoto}, see formula 5.35 and lemma $5.1$. It is also useful for the next definition the lemma $5.2.1$ in Guleska \cite{Lok04}.

\begin{defi}
Let  $l,p,q \in \tfrac{1}{2} \mathbb{Z}$, $l \geq 0$,  such that $p,q \in [-l,l]$,  $p \equiv q \equiv l \text{ mod } 1$. The cusp form  $F_{p,q}^l: \Gamma \backslash G \longrightarrow \C$ associated to the value   $\widehat{\lambda}$ (with  $\tfrac{1}{4} + r^2 = \widehat{\lambda}$) is given for the following formula:
\begin{equation}\label{FC}
 F_{p,q}^l \big(  z+\lambda j,  K \big) \, = \, F_{p,q}^l \big( n[z]a[\lambda]  K \big)  \, = \,\sum_{0 \neq w \in \Lambda}   \sum_{\substack{ a=-l \\ a \in  \rfrac{1}{2} \mathbb{Z} \\ a \equiv l \text{ mod } 1 }}^l c(w) \, j_a^w (ir ; \lambda ) \, e^{2 \pi i \, \Re(wz) } \; \Phi_{aq}^l (K),
\end{equation} 
\end{defi}
where  
\begin{equation} \label{FC0}
 j_a^w (s ; \lambda ) \, := \, 2 \, (-1)^{l-p} \; \pi^{s} \, \abs{ w }^{s-1} \, \bigg( \frac{i w}{\abs{w}}  \bigg)^{-p-a} \cdot \mathcal{W}_a^l (s,p; \tfrac{\abs{w} \lambda}{2} ),
\end{equation}
moreover
\begin{equation} \label{FC01}
\mathcal{W}_a^l (s,p;  \lambda ) \, := \,  \sum_{ v=0 }^{ l- \frac{1}{2}( \abs{a+p}+\abs{a-p})} (-1)^v \; \xi_p^l (a,v) \, \frac{\big(2\pi  \lambda \big)^{l+1-v}}{\Gamma(1+l+s-v)} \cdot K_{s+l-\abs{a+p}-v} \big( 4 \pi  \lambda \big).   
\end{equation}

\bigskip
%we will normalize and suppose that the $c(w)$ coefficients satisfy the classic multiplicative relations.
The associated L-function to the cusp form $F_{p,q}^l$ twisted by  $\chi_m$, an Hecke character for  $K_D$,  is given by the formula:
\[
     L(s,F_{p,q}^l,\chi_{m} )  \, = \,  \sum_{  (w)  }     \frac{ c(w) \cdot \chi_{m} \big((w)\big)  }{ \abs{w}^{2s}  } \, = \,
       \prod_{ \substack{ p \in Z[i]  \\  p  \text{ primo}  } } \Bigg( 1 - \frac{ c(p) \cdot \chi_m(p) }{ \abs{p}^{2s} } + \frac{ \chi_m(p)^2}{ \abs{p}^{4s}} \Bigg)^{-1},
\] 
the sum is over all the integral ideals $(w)$ not zero of $Z[i]$.
 
\bigskip
The cusp form in (\ref{FC}) for $K=I$ is given by
\begin{equation}\label{FC02}
F_{p,q}^l \big(  z+\lambda j, I \big)
   \, = \,  \sum_{0 \neq w \in \Lambda} \sum_{\substack{ a=-l \\ a \in  \rfrac{1}{2} \mathbb{Z} \\ a \equiv l \text{ mod } 1 }}^l c(w) \, j_a^w (ir ; \lambda ) \, e^{2 \pi i \, \text{Re}(wz) } \; \Phi_{aq}^l (I)   \, = \,  \sum_{0 \neq w \in \Lambda'}      c(w) \, j_q^w (ir ; \lambda ) \, e^{2 \pi i \, \Re(wz) }.
\end{equation}

Later, of (\ref{FC02}) and (\ref{FC}) it is seen that the following equation is valid for $K \in \SU(2)$ arbitrary  \begin{equation}\label{FC1}
 F_{p,q}^l \big(  z+\lambda j,  K \big) \, = \, \sum_{\substack{ a=-l \\ a \in  \rfrac{1}{2} \mathbb{Z} \\ a \equiv l \text{ mod } 1 }}^l \Phi_{aq}^l (K) \cdot F_{p,a}^l \big(  z+\lambda j, I \big).
\end{equation}

\bigskip
So for the identity in (\ref{Medida}) we have that $   \Big(  F_{p,q}^l  ,  d\epsilon_{it}  \Big) \hspace{10cm} $
\begin{align}
   & \, = \, \int_{\Gamma \backslash G }   F_{p,q}^l \big(   z+\lambda j, K \big)  \cdot  E \big( z+ \lambda j, it \big)
 \;  \sum_{\substack{ L \in  \mathbb{Z} \\ L \geq 0 }}  \; \; \sum_{\substack{ k,m=-\rfrac{L}{2} \\ k \equiv m \equiv \rfrac{L}{2} \text{ mod } 1 }}^{\rfrac{L}{2}}  (-1)^{m-k} \cdot \mathcal{T}^{\, L}_{kk} (I) \cdot  {E}_{-m,-k}^{\rfrac{L}{2}} \big(z +\lambda j,-it \big)    \; \cdot  \; \overline{\mathcal{T}_{km}^L (K)} \, d g 
 && \nonumber \\  
          & \, = \,  \int_{ \Gamma \backslash \mathbb{H}^3 }  E \big( z+ \lambda j,  it \big)
   \sum_{\substack{ a=-l \\ a \in  \rfrac{1}{2} \mathbb{Z} \\ a \equiv l \text{ mod } 1 }}^l    \sum_{\substack{ L \in  \mathbb{Z} \\ L \geq 0 }}  \; \; \sum_{\substack{ k,m=-\rfrac{L}{2} \\ k \equiv m \equiv \rfrac{L}{2} \text{ mod } 1 }}^{\rfrac{L}{2}}   (-1)^{m-k} \cdot \mathcal{T}^{\, L}_{kk} (I)  \cdot  {E}_{-m,-k}^{\rfrac{L}{2}} \big(z +\lambda j,-it\big)   \cdot F_{p,a}^l \big(  z+\lambda j, I \big) &&  \nonumber \\ 
         & \hspace{6cm} \Bigg[ \int_{\SU(2) }  \Phi_{aq}^l (K)   \cdot   \overline{\mathcal{T}_{km}^L (K)} \; d k \Bigg]  d V. \hspace{4cm} \text{by } (\ref{FC1})  && \label{Fcuspidal} 
 \end{align}

\bigskip
Seen the integral between brackets, for the identities (\ref{T1}), (\ref{A11}) and (\ref{A10}) it is true that 
\begin{equation}\label{Fcusp2} 
 \int_{\SU(2) }  \Phi_{aq}^l (K)  
   \cdot   \overline{\mathcal{T}_{km}^L (K)} \; d k  \, = \,
    (-1)^{m-k} \cdot \mathcal{T}_{qq}^{2l} (I)   \cdot  \frac{\sqrt{(l+q)!(l-q)!}}{\sqrt{(l+a)!(l-a)!}} \cdot \frac{1}{2l+1} \cdot \delta_{l,\frac{L}{2}} \,  \delta_{a,m} \, \delta_{q,k}.
\end{equation}

 \bigskip
Replacing (\ref{Fcusp2}) in (\ref{Fcuspidal}) we obtain that %$ \abs{\Lambda_D} \cdot  \Big(  F_{p,q}^l  ,  d\epsilon_{it}  \Big) $
 %$$  \, = \, 
%\int_{ \Gamma \backslash \mathbb{H}^3 }  E \big( z+ \lambda j, it \big)
 % \sum_{\substack{ a=-l \\ a \in  \rfrac{1}{2} \mathbb{Z} \\ a \equiv l \text{ mod } 1 }}^l    \sum_{\substack{ L \in  \mathbb{Z} \\ L \geq 0 }}  \sum_{\substack{ k,m=-\rfrac{L}{2} \\ k \equiv m \equiv \rfrac{L}{2} \text{ mod } 1 }}^{\rfrac{L}{2}}  \, (-1)^{m-k} \cdot \mathcal{T}^{\, l}_{kk} (I)  \cdot  {E}_{-m,-k}^{\rfrac{L}{2}} \big(z +\lambda j,-it\big)  \cdot F_{p,a}^l \big(  z+\lambda j, I \big) $$ $$ \cdot \,
% (-1)^{m-k} \cdot \mathcal{T}_{qq}^{2l} (I)   \cdot  \frac{\sqrt{(l+q)!(l-q)!}}{\sqrt{(l+a)!(l-a)!}} \cdot \frac{1}{2l+1} \cdot \delta_{l,\frac{L}{2}} \,  \delta_{a,m} \, \delta_{q,k}   \,  d V $$
\begin{equation}\label{Im1}
 \Big(  F_{p,q}^l  ,  d\epsilon_{it}  \Big) \, = \, \frac{1}{2 \pi^2}  \sum_{\substack{ a=-l \\ a \in  \rfrac{1}{2} \mathbb{Z} \\ a \equiv l \text{ mod } 1 }}^l 
  \frac{\sqrt{(l+q)!(l-q)!}}{\sqrt{(l+a)!(l-a)!}} \,
\int_{ \Gamma \backslash \mathbb{H}^3 }  E \big( z+ \lambda j,  it \big) \cdot
    {E}_{-a,-q}^{l} \big(z +\lambda j,-it\big)  \cdot F_{p,a}^l \big(  z+\lambda j, I \big)  \,  d V .
\end{equation}

\bigskip
We make the following change of variable 
$$ \sigma ( z + \lambda j ) \, = \, z'+\lambda' j \; \; \iff \; \;  z + \lambda j  \, = \, \sigma^{-1} ( z'+\lambda' j).$$
Therefore
$$ n[z]a[\lambda] \, = \, \sigma^{-1} \, n[z']a[\lambda'] \, A^{-1}, $$
where $A:= T \big(  \sigma^{-1} \, n[z']a[\lambda'] \big)$. Then the inner product is given by 
$$ \frac{1}{2 \pi^2}   \sum_{\substack{ a=-l \\ a \in  \rfrac{1}{2} \mathbb{Z} \\ a \equiv l \text{ mod } 1 }}^l \frac{\sqrt{(l+q)!(l-q)!}}{\sqrt{(l+a)!(l-a)!}} \hspace{8cm} $$
$$  \int_{ \sigma (   \Gamma \backslash \mathbb{H}^3 ) }  E \big( \sigma^{-1} ( z'+\lambda' j), it \big) \,
    \sum_{\sigma\,\in\,_{\scriptstyle\Gamma'_\infty}\backslash\Gamma}   \overline{ D_{-a,-q}^l \big(  \Phi_{ T(  n[z']a[\lambda']  A^{-1} )^{-1}} \big)} \; \Im \, (  z'+\lambda' j)^{1-it}  \cdot  F_{p,a}^l \big(  \sigma^{-1} \, n[z']a[\lambda'] \, A^{-1} \big) 
 \; \frac{dx' dy' d\lambda'}{\lambda'^3}.$$

\bigskip
Now, for the invariance of the Eisenstein series and the functions $ F_{p,a}^l $  the inner product is
\begin{equation}\label{Tm2.1}
 \frac{1}{2 \pi^2}   \sum_{\substack{ a=-l \\ a \in  \rfrac{1}{2} \mathbb{Z} \\ a \equiv l \text{ mod } 1 }}^l  \frac{\sqrt{(l+q)!(l-q)!}}{\sqrt{(l+a)!(l-a)!}}
 \int_{ \sigma (   \Gamma \backslash \mathbb{H}^3 ) }  E \big(  z'+\lambda' j, it \big) \,     \sum_{\sigma\,\in\,_{\scriptstyle\Gamma'_\infty}\backslash\Gamma}   \overline{ D_{-a,-q}^l \big(  \Phi (A) \big) } \; \lambda'^{1-it}  \cdot  F_{p,a}^l \big(   n[z']a[\lambda'] \, A^{-1} \big) 
 \; \frac{dx' dy' d\lambda'}{\lambda'^3}.
\end{equation}

\bigskip
For (\ref{FC1})
\begin{equation} \label{Tm3}
 F_{p,a}^l \big(   n[z']a[\lambda'] \, A^{-1} \big)  \, = \,
  \sum_{\substack{ u=-l \\ u \in  \rfrac{1}{2} \mathbb{Z} \\ u \equiv l \text{ mod } 1 }}^l  \Phi_{ua}^l \big( A^{-1} \big) \cdot F_{p,u}^l \big(  z'+\lambda' j, I \big).
\end{equation}

\bigskip
Replacing (\ref{Tm3}) in (\ref{Tm2.1}), the inner product is given by
$$ \; \; \; \frac{1}{2 \pi^2}   \sum_{\substack{ a=-l \\ a \in  \rfrac{1}{2} \mathbb{Z} \\ a \equiv l \text{ mod } 1 }}^l  \frac{\sqrt{(l+q)!(l-q)!}}{\sqrt{(l+a)!(l-a)!}} \sum_{\substack{ u=-l \\ u \in  \rfrac{1}{2} \mathbb{Z} \\ u \equiv l \text{ mod } 1 }}^l  \;  \sum_{\sigma\,\in\,_{\scriptstyle\Gamma'_\infty}\backslash\Gamma}   \hspace{9cm} $$
$$  \int_{ \sigma (   \Gamma \backslash \mathbb{H}^3 ) }  E \big(  z'+\lambda' j, it \big) \cdot   \overline{ D_{-a,-q}^l \big(  \Phi (A) \big) } \cdot \lambda'^{1- it}  \cdot  \Phi_{ua}^l \big( A^{-1} \big) \cdot F_{p,u}^l \big(  z'+\lambda' j, I \big)
 \; \frac{dx' dy' d\lambda'}{\lambda'^3} $$
 $$ = \,  \frac{1}{2 \pi^2}  \sum_{\substack{ u=-l \\ u \in  \rfrac{1}{2} \mathbb{Z} \\ u \equiv l \text{ mod } 1 }}^l \, \sum_{\sigma\,\in\,_{\scriptstyle\Gamma'_\infty}\backslash\Gamma}  \hspace{13cm} $$
\begin{equation}\label{FC8}
 \int_{ \sigma (   \Gamma \backslash \mathbb{H}^3 ) }  E \big(  z+\lambda j, it \big) \cdot F_{p,u}^l \big(  z+\lambda j, I \big) \cdot \lambda^{1-it}  \Bigg[
   \sum_{\substack{ a=-l \\ a \in  \rfrac{1}{2} \mathbb{Z} \\ a \equiv l \text{ mod } 1 }}^l  \frac{\sqrt{(l+q)!(l-q)!}}{\sqrt{(l+a)!(l-a)!}} \cdot  \overline{ D_{-a,-q}^l \big(  \Phi (A) \big) }  \cdot  \Phi_{ua}^l \big( A^{-1} \big)  \Bigg] \frac{dx dy d\lambda}{\lambda^3}.
\end{equation}

\bigskip
Now we solve the summation between brackets, using (\ref{A11}), (\ref{A8}) and (\ref{WI10}) we obtain 
\begin{equation}\label{FC9} 
 \sum_{\substack{ a=-l \\ a \in  \rfrac{1}{2} \mathbb{Z} \\ a \equiv l \text{ mod } 1 }}^l   \frac{\sqrt{(l+q)!(l-q)!}}{\sqrt{(l+a)!(l-a)!}} \cdot   \overline{ D_{-a,-q}^l \big(  \Phi (A) \big) }  \cdot  \Phi_{ua}^l \big( A^{-1} \big) \, = \,  \delta_{uq}.
\end{equation}

\bigskip
Then, replacing (\ref{FC9}) in (\ref{FC8}) we have 
\begin{align}
 \Big(  F_{p,q}^l  ,  d\epsilon_{it}  \Big) 
  & \, = \,  \frac{1}{2 \pi^2}   \sum_{\sigma\,\in\,_{\scriptstyle\Gamma'_\infty}\backslash\Gamma} \; \int_{ \sigma (   \Gamma \backslash \mathbb{H}^3 ) }  E \big(  z+\lambda j, it \big)  \cdot  F_{p,q}^l \big(  z+\lambda j, I \big)  \cdot  \lambda^{1-it}   \; \frac{dx dy d\lambda}{\lambda^3} &&    \nonumber \\  
        & \, = \,   \frac{1}{2 \pi^2}   \int_{ \Gamma'_\infty \backslash   \mathbb{H}^3   }  E \big(  z+\lambda j,it \big)  \cdot  F_{p,q}^l \big(  z+\lambda j, I \big)  \cdot  \lambda^{1-it}   \; \frac{dx dy d\lambda}{\lambda^3}. &&    \label{FC10}
 \end{align}

\bigskip
For (\ref{FC02}) and (\ref{FC0})
\begin{equation}\label{Aux8}
 F_{p,q}^l \big(  z+\lambda j, I \big)  \, = \,
    \sum_{0 \neq w \in \Lambda} 2 \,  (-1)^{l-p} \, \pi^{ir} \, c(w) \, \abs{ w }^{ir-1} \,  \bigg( \frac{i w}{\abs{w}}  \bigg)^{-p-q} \cdot  \mathcal{W}_q^l (ir,p; \tfrac{\abs{w} \lambda}{2} ) \cdot e^{2 \pi i \, \Re(wz) } .
\end{equation}

\bigskip
Replacing (\ref{Aux8}) in (\ref{FC10})
$$  \Big(  F_{p,q}^l  ,  d\epsilon_{it}  \Big) \, = \, \frac{1}{ \pi^2} \sum_{0 \neq w \in \Lambda} (-1)^{l-p} \, \pi^{ir} \, c(w) \, \abs{ w }^{ir-1} \,  \bigg( \frac{i w}{\abs{w}}  \bigg)^{-p-q} \hspace{5cm}  $$
\begin{equation}\label{FC11}
 \hspace{3cm} \cdot \int_0^\infty     \mathcal{W}_q^l (ir,p; \tfrac{\abs{w} \lambda}{2} ) \cdot  \lambda^{1-it}  \; \Bigg[ \int_{ [0,1]^2 }  E \big(  z+\lambda j,it \big)  \cdot   \, e^{2 \pi i \, \Re(wz) }   \; dx dy    \Bigg] \;  \frac{ d\lambda}{\lambda^3}.
\end{equation}

\bigskip
\bigskip
The Fourier expansion of the clasical Eisenstein series evaluated in $s=it$ (with  $t \neq 0$) is given by:
$$    E (z+\lambda j,it) \, = \, [\Gamma_{\infty} : \Gamma'_{\infty} ]  \; \lambda^{1+it} \;  - \;   \frac{ i \pi  }{t} \;
\frac{ \zeta_K ( it )}{ \zeta_K ( 1+it )}  \cdot \lambda^{1-it}  \hspace{4cm} $$
$$ \hspace{2cm} + \;     \frac{(2\pi)^{1+it} \,  \lambda}{ \Gamma(1+it)}  \,  \sum_{ 0 \neq \alpha \in \Lambda'} 
\mathcal{D}_{0} (-\alpha; it) \cdot \abs{ \alpha }^{it}  \cdot  K_{it} \big( 4 \pi \abs{\alpha} \lambda \big) \cdot e^{-4 \pi i \, \Re(\alpha z )}. $$

\bigskip
Then
%$$  \int_{P_\Lambda    }  E \big(  z+\lambda j,it \big)  \cdot   \, e^{4 \pi i \, \Re(wz) }   \; dx dy      \, = \, 
% \frac{1}{ \abs{ \Lambda_D} \cdot [\Gamma_{\infty} : \Gamma'_{\infty} ] } \,    \frac{2\pi^{1+it} \,  \lambda}{\Gamma(1+it)}  \,  \sum_{ 0 \neq \alpha \in \Lambda'} 
%\mathcal{D}_{0} (-\tfrac{\alpha}{2}; it) \cdot \abs{ \alpha }^{it}  \cdot  K_{it} \big( 2 \pi \abs{\alpha} \lambda \big) $$ $$ \cdot
%\int_{P_\Lambda} e^{-2 \pi i \, \text{ Re} (\alpha z)}  \cdot    e^{4 \pi i \, \Re(wz) }   \; dx dy  $$
\begin{equation}\label{FC12}
 \int_{[0,1]^2 }  E \big(  z+\lambda j,it \big)  \cdot   \, e^{2 \pi i \, \Re(wz) }   \; dx dy      \, = \,   \frac{2  \pi^{1+it} \,  \lambda}{\Gamma(1+it)}  
\, \mathcal{D}_{0} (-\tfrac{w}{2}; it) \cdot \abs{ w }^{it} 
 \cdot  K_{it} \big( 2 \pi \abs{w} \lambda \big).
\end{equation}

\bigskip
Replacing (\ref{FC12}) in (\ref{FC11}) is followed that
%$$  \Big(  F_{p,q}^l  ,  d\epsilon_{it}  \Big) \, = \, \frac{\sqrt{l+1}}{2 \pi^2 \, \sqrt{2l+1}  \, \abs{\Lambda_D}} \sum_{0 \neq w \in \Lambda'} (-1)^{l-p} \cdot (2\pi)^{ir} \cdot c(w) \cdot \abs{ w }^{ir-1} \cdot  \bigg( \frac{i w}{\abs{w}}  \bigg)^{-p-q}   $$ $$ \cdot \int_0^\infty     \mathcal{W}_q^l (ir,p; \abs{w} \lambda ) \cdot  \lambda^{1-it}  \; \Bigg[  \frac{1}{ [\Gamma_{\infty} : \Gamma'_{\infty} ] } \,    \frac{2\pi^{1+it} \,  \lambda}{\Gamma(1+it)}  \,  \sum_{ 0 \neq \alpha \in \Lambda'} 
%\mathcal{D}_{0} (-\tfrac{\alpha}{2}; it) \cdot \abs{ \alpha }^{it}  \cdot  K_{it} \big( 2 \pi \abs{\alpha} \lambda \big) \cdot   \delta_{\alpha, 2w} \Bigg] \;  \frac{ d\lambda}{\lambda^3} $$
% segunda
%$$ \, = \, \frac{\sqrt{l+1}}{2 \pi^2 \, \sqrt{2l+1} } \; \frac{1}{  [\Gamma_{\infty} : \Gamma'_{\infty} ] \cdot  \abs{\Lambda_D} }  \,  \,    \frac{2\pi^{1+it} }{\Gamma(1+it)} \;  (-1)^{l-p} \cdot (2\pi)^s  \cdot i^{-p-q}   \Bigg[ \sum_{0 \neq w \in \Lambda'}  \sum_{ 0 \neq \alpha \in \Lambda'} c(w) \cdot \abs{ w }^{s-1} \cdot \abs{ \alpha }^{it} \cdot \mathcal{D}_{0} (-\tfrac{\alpha}{2}; it) \cdot \bigg( \frac{ w}{\abs{w}}  \bigg)^{-p-q}  \,  \Bigg] $$
%$$ \cdot  \int_0^\infty     \mathcal{W}_q^l (s,p; \abs{w} \lambda ) \cdot  K_{it} \big( 2 \pi \abs{\alpha} \lambda \big)  \cdot \lambda^{1-it} \cdot \lambda  \cdot   \delta_{\alpha, 2w} \; \frac{ d\lambda}{\lambda^3} $$
$$  \Big(  F_{p,q}^l  ,  d\epsilon_{it}  \Big) \, = \,   \frac{2  \pi^{-1+it+ir} }{\Gamma(1+it)}  \, (-1)^{l-p} \, i^{-p-q} \, \Bigg[ \sum_{0 \neq w \in \Lambda}   \, c(w) \, \abs{ w }^{-1+ir+it} \, \mathcal{D}_{0} (-\tfrac{w}{2}; it)  \;  \bigg( \frac{w}{\abs{w}}  \bigg)^{-p-q}
\Bigg]  \hspace{4cm}  $$
\begin{equation}\label{FC13}
\cdot  \int_0^\infty     \mathcal{W}_q^l (ir,p; \tfrac{\abs{w} \lambda}{2} ) \cdot  K_{it} \big( 2 \pi \abs{w} \lambda \big)  \cdot \lambda^{-1-it}  \;  d\lambda.
\end{equation}

\bigskip
Nevertheless for (\ref{FC01})
\begin{equation}\label{FC14}
\mathcal{W}_q^l (ir,p; \tfrac{\abs{w} \lambda}{2} ) \, = \,  \sum_{ v=0 }^{ l- \frac{1}{2}( \abs{q+p}+\abs{q-p})} (-1)^v \; \xi_p^l (q,v) \, \frac{\big(\pi  \abs{w} \lambda \big)^{1+l-v}}{\Gamma(1+l-v+ir)} 
\cdot K_{l-\abs{q+p}-v+ir} \, ( 2 \pi \abs{w} \lambda ). 
\end{equation}

\bigskip
Replacing the equality (\ref{FC14}) in (\ref{FC13}) we obtain that 
$$  \Big(  F_{p,q}^l  ,  d\epsilon_{it}  \Big) \, = \,   \frac{2 \pi^{l+it+ir} }{\Gamma(1+it)}  \, (-1)^{l-p} \, i^{-p-q} \,   \sum_{ v=0 }^{ l- \frac{1}{2}( \abs{q+p}+\abs{q-p})}  \;  \frac{  (-\pi)^{-v} \, \xi_p^l (q,v)}{\Gamma(1+l-v+ir)}  $$
$$ \Bigg[ \sum_{0 \neq w \in \Lambda}   \, c(w) \,  \abs{ w }^{l-v+ir+it} \, \mathcal{D}_{0} (-\tfrac{w}{2}; it)  \;  \bigg( \frac{w}{\abs{w}}  \bigg)^{-p-q}
\Bigg]  \;   \int_0^\infty \lambda^{l-v-it} \cdot  K_{l-\abs{q+p}-v+ir} \, ( 2 \pi \abs{w} \lambda ) \cdot  K_{it} \big( 2 \pi \abs{w} \lambda \big)  \;  d\lambda.
$$

\bigskip
Changing the variable $x = 2 \pi \abs{w} \lambda$, 
$$  \Big(  F_{p,q}^l  ,  d\epsilon_{it}  \Big) \, = \, 
  \frac{2^{-l+it} \, \pi^{-1+ir+2it} }{\Gamma(1+it)} 
   \, (-1)^{l-p} \, i^{-p-q} \,   \sum_{ v=0 }^{ l- \frac{1}{2}( \abs{q+p}+\abs{q-p})}  \;  \frac{  (-2)^{v} \, \xi_p^l (q,v)}{\Gamma(1+l-v+ir)}  $$
   \begin{equation}\label{FC15}
\Bigg[ \sum_{0 \neq w \in \Lambda}   \, c(w) \,  \abs{ w }^{-1+ir+2it} \, \mathcal{D}_{0} (-\tfrac{w}{2}; it)  \;  \bigg( \frac{w}{\abs{w}}  \bigg)^{-p-q}
\Bigg]  \;   \int_0^\infty x^{l-v-it} \cdot  K_{l-\abs{q+p}-v+ir} \, ( x ) \cdot  K_{it} ( x )  \;  dx.
\end{equation}

\bigskip
\bigskip
The integral at the end of the last expression can be found in  \cite{GR} page 692.
$$ \int_0^\infty x^{\,-z} \cdot K_{\mu} ( x )  \cdot K_{\nu}   ( x ) \;  dx   \, = \,  \frac{2^{-2-z}}{\Gamma(1-z)} \Gamma \big( \tfrac{1-z+\mu+\nu}{2}  \big) \Gamma \big( \tfrac{1-z-\mu+\nu}{2}  \big) \Gamma \big( \tfrac{1-z+\mu-\nu}{2}  \big) \Gamma \big( \tfrac{1-z-\mu-\nu}{2}  \big) \cdot F \big( \tfrac{1-z+\mu+\nu}{2} , \tfrac{1-z-\mu+\nu}{2} ;1-z ; 0 \big), $$
for $\Re (z) <  1- \abs{\Re(\mu)} - \abs{\Re(\nu)}$. Particularly,  

$$\int_0^\infty   x^{l-v-it} \cdot K_{l-\abs{q+p}-v+ir} \, ( x ) \cdot  K_{it} \big( x \big) \,  dx \,  = \,  \frac{2^{l-2-v-it}}{\Gamma(l+1-v-it)}$$
\begin{equation}\label{FC16}
 \cdot \; 
   \Gamma \big( \tfrac{1}{2}+l-v-\tfrac{1}{2} \abs{q+p} +\tfrac{ir}{2} \big) \cdot
    \Gamma \big( \tfrac{1}{2}+\tfrac{1}{2} \abs{q+p} -\tfrac{ir}{2} \big) \cdot 
  \Gamma \big( \tfrac{1}{2}+l-v-\tfrac{1}{2} \abs{q+p} -it + \tfrac{ir}{2}  \big) \cdot 
    \Gamma   \big(  \tfrac{1}{2}+\tfrac{1}{2} \abs{q+p} -it - \tfrac{ir}{2} \big).
\end{equation}
%para $s \in \mathbb{C}$ such that  $ -l+v  <1- \abs{ \Re (s)+l-\abs{q+p}-v }$.

\bigskip
On the other hand, for (\ref{IE20}) 
\begin{equation}\label{FC16.4}
 S \, := \,   \sum_{0 \neq w \in \Lambda}  c(w) \, \abs{ w }^{-1+ir+2it} \, \mathcal{D}_{0} (-\tfrac{w}{2}; it) \, \bigg( \frac{ w}{\abs{w}}  \bigg)^{-p-q}   \, = \,   \frac{4^2}{\zeta_K (1+it)} \sum_{0 \neq w \in \Lambda}  c(w) \, \abs{ w }^{-1+ir+2it} \, \bigg( \frac{ w}{\abs{w}}  \bigg)^{-p-q} \, \sigma_{-it} (w,0).
\end{equation}

\bigskip
However, for the lemma (\ref{LFC}) 
\begin{equation}\label{FC16.5}
 \sum_{0 \neq  w \in Z[i]}  c(w) \cdot \abs{w}^s \cdot  \bigg( \frac{ w }{\abs{w}}  \bigg)^{\alpha} \cdot \sigma_{\nu} (w,0) \, = \,  \frac{1}{4}
 \, \frac{L(-\tfrac{s}{2}, \phi, \chi_\alpha ) \, L(-\tfrac{s}{2}-\nu, \phi, \chi_\alpha ) }{ L( -s-\nu, \chi_{2\alpha}) } .
\end{equation}

\bigskip
Then, replacing (\ref{FC16.5}) in (\ref{FC16.4}) it follows that
\begin{equation}\label{FC17}
S \, = \,  \frac{ 4 }{\zeta_K (1+it)} \cdot  \frac{L(\tfrac{1}{2}-\tfrac{ir}{2}-it, F_{p,q}^l, \chi_{-p-q} ) \cdot L(\tfrac{1}{2}-\tfrac{ir}{2}, F_{p,q}^l, \chi_{-p-q} ) }{ L( 1-ir-it, \chi_{-2p-2q}) }.
\end{equation}

\bigskip
Finally, replacing (\ref{FC16}) and (\ref{FC17}) in (\ref{FC15}) we obtain
$$  \Big(  F_{p,q}^l  ,  d\epsilon_{it}  \Big) \, = \, 
  (-1)^{l-p} \, i^{-p-q} \,  \frac{ \pi^{-1+ir+2it} }{\Gamma(1+it)}  $$
  $$   \frac{L(\tfrac{1}{2}-\tfrac{ir}{2}-it, F_{p,q}^l, \chi_{-p-q} ) \cdot L(\tfrac{1}{2}-\tfrac{ir}{2}, F_{p,q}^l, \chi_{-p-q} ) \cdot  \Gamma \big( \tfrac{1}{2}+\tfrac{1}{2} \abs{q+p} -\tfrac{ir}{2} \big)  \cdot  \Gamma   \big(  \tfrac{1}{2}+\tfrac{1}{2} \abs{q+p} -\tfrac{ir}{2}-it \big) }{ \zeta_K (1+it) \cdot L( 1-ir-it, \chi_{-2p-2q}) }   $$
  \begin{equation}\label{FC18}
  \sum_{ v=0 }^{ l- \frac{1}{2}( \abs{q+p}+\abs{q-p})} (-1)^{v} \; \xi_p^l (q,v) 
   \cdot   \frac{ 
   \Gamma \big( \tfrac{1}{2}+l-v-\tfrac{1}{2} \abs{q+p} +\tfrac{ir}{2} \big) \cdot
       \Gamma \big(  \tfrac{1}{2}+l-v-\tfrac{1}{2} \abs{q+p} +\tfrac{ir}{2}-it   \big)    }{    \Gamma(1+l-v+ir) \cdot \Gamma(1+l-v-it)}. 
  \end{equation}

%\bigskip
%Por la f\'ormula de Stirling $ \abs{\Gamma (\sigma \pm t)} \, \sim \, 
%e^{-\tfrac{\pi}{2} \abs{t}} \, \abs{t}^{\, \sigma-\tfrac{1}{2}}$. Aplicando dicha f\'ormula 
%$$ \abs{ \Gamma   \big(  \tfrac{1}{2}+\tfrac{1}{2} \abs{q+p}-i(t+\tfrac{r}{2} )  \big) } \, \sim \, e^{-\tfrac{\pi}{2} \big(t+\tfrac{r}{2} \big)} \, \abs{t}^{\, \tfrac{1}{2} \abs{q+p} }.$$
%$$ \abs{ \Gamma \big( l+ \tfrac{1}{2}-v-\tfrac{1}{2} \abs{q+p} - i(t-\tfrac{r}{2} )   \big) } \, \sim \, e^{-\tfrac{\pi}{2} \big(t-\tfrac{r}{2} \big)} \, \abs{t}^{\, l-v-\tfrac{1}{2} \abs{q+p} }.$$
%$$ \abs{ \Gamma(l+1-v-it) } \, \sim \, e^{-\tfrac{\pi}{2} t} \, \abs{t}^{\, l+\tfrac{1}{2}-v}, \hspace{2.5cm}
 %\abs{ \Gamma(1+it) } \, \sim \, e^{-\tfrac{\pi}{2} t} \, \abs{t}^{\,\tfrac{1}{2}}.$$

%\bigskip
%then it seems that
%$$ \frac{\abs{ \Gamma   \big(  \tfrac{1}{2}+\tfrac{1}{2} \abs{q+p}-i(t+\tfrac{r}{2} )  \big) } \cdot \abs{ \Gamma \big( l+ \tfrac{1}{2}-v-\tfrac{1}{2} \abs{q+p} - i(t-\tfrac{r}{2} )   \big) }  }{  \abs{ \Gamma(l+1-v-it) } } \, \sim \, 
%\frac{ e^{\tfrac{\pi}{2} \big(t+\tfrac{r}{2} \big)} \, \abs{t}^{\, \tfrac{1}{2} \abs{q+p} } \; e^{\tfrac{\pi}{2} \big(t-\tfrac{r}{2} \big)} \, \abs{t}^{\, l-v-\tfrac{1}{2} \abs{q+p} } }{ e^{\tfrac{\pi}{2} t} \, \abs{t}^{\, l+\tfrac{1}{2}-v}}  \sim \, \frac{ e^{\tfrac{\pi}{2} t} }{\abs{t}^{\tfrac{1}{2}}} .$$

\bigskip
For the Stirling's formula $ \abs{\Gamma (\sigma \pm it)} \, \sim \, \sqrt{2 \pi} \;
e^{-\tfrac{\pi}{2} \abs{t}} \, \abs{t}^{\, \sigma-\tfrac{1}{2}}$ as $ t \longrightarrow \infty$. Then, the absolute value of gamma factors in (\ref{FC18}) where appears $t$ satisfy that
\begin{equation}\label{FC19}
  \ll \, t^{-1} \; \; \text{ as } t \longrightarrow \infty.
\end{equation}

\bigskip
It is known that the Dedekind zeta function in (\ref{FC18}) is estimated as follows:
\begin{equation}\label{FC20}
  t^{-\epsilon} \ll \abs{\zeta_K (1+it)} \ll t^\epsilon, \; \; \forall \epsilon > 0.
\end{equation}

\bigskip
The results of Michel and Venkatesh \cite{MiVen} and Garrett and Diaconu \cite{GaDia} proved that  exists $\delta >0$ such that for every $\epsilon > 0$ 
% (as in Koyama \cite{Koy} for Maass cusp forms)
\begin{equation}\label{FC21}
  L(\tfrac{1}{2}-it, F_{p,q}^l, \chi_\alpha ) \, {\ll}_{F_{p,q}^l,\epsilon,\chi_\alpha} \; \abs{t}^{1-\delta+\epsilon}  \; \; \text{ as } t \longrightarrow \infty.
\end{equation}

The estimates (\ref{FC19}), (\ref{FC20}) and (\ref{FC21}) in the formula (\ref{FC18}) show the next:
\begin{prop}\label{PFC}
We have that % Be $l,p,q \in \tfrac{1}{2} \mathbb{Z}$, $l \geq 0$,  y   tales  que  $p,q \in [-l,l]$,  $p \equiv q \equiv l \text{ mod } 1$. Entonces 
\[
\lim_{t \rightarrow \infty}  \Big(  F_{p,q}^l  ,  d\epsilon_{it}  \Big) \, = \, 0.
\]
%where $F_{p,q}^l$ is the associated cuspid formula (see definition (\ref{FC})).
\end{prop}

%\newpage
\bigskip
\bigskip
\section{Incomplete Eisenstein Series}
Let $\psi(\lambda) \in C_0^\infty ((0,\infty))$ be a rapidly decreasing function at $0$ and $\infty$. Consider the Mellin transform of $\psi$
\begin{equation}\label{H}
H(s) \, = \, \int_0^\infty \psi(\lambda) \, \lambda^{-s} \, \frac{d \lambda}{\lambda}.
\end{equation}
$H(s)$ is of Schwartz class in $t$ for each vertical line $\sigma +it$, we denote such line by $(\sigma)$.

\bigskip
The Mellin inversion formula affirms that
\begin{equation}\label{IM}
\psi(\lambda)  \, = \, \frac{1}{2\pi i} \int_{(\sigma)} H(s) \, \lambda^s \, ds
\end{equation}
for any $\sigma \in \R$.

\begin{defi}\label{DIES}
Let  $l,a,b \in \tfrac{1}{2} \mathbb{Z}$, $l \geq 0$, such that  $l \equiv a \equiv b \text{ mod } 1$ with   $a,b \in [-l,l]$, the incomplete Eisenstein series associated to $\psi$, denoted $ F_{a,b}^l (\psi)$, are given by
$$ F_{a,b}^l (\psi) (g)   \, := \,  \sum_{\sigma \,\in \,_{\scriptstyle\Gamma'_\infty}\backslash\Gamma}   \overline{ D_{ab}^l \big(  \Phi_{ T( \sigma g )^{-1}} \big)} \; \psi \big( \Im \, \sigma g ( j ) \big) \, = \, \frac{1}{2 \pi i} \, 
 \int_{(3)}  H(s)  \cdot E_{ab}^l (g,s-1) \; ds.$$
\end{defi}

\begin{rem}
It will be used too the notation $ F_{a,b}^l (\psi) (g) \, = \,F_{a,b}^l (\psi) (z + \lambda j, K).  $
\end{rem}

For  (\ref{PropE})
\begin{equation}\label{IES0}
 F_{a,b}^l (\psi) (z + \lambda j, K) \, = \, \frac{1}{2 \pi i} \, 
\sum_{\substack{ r=-l \\ r \in \tfrac{1}{2} \Z \\ r \equiv l \text{ mod } 1 }}^{l}   \overline{D^l_{ar} \big( \Phi (K)^{-1}  \big)}  \int_{(3)}  H(s)  \cdot E_{rb}^l (z+ \lambda j ,s-1) \; ds.
\end{equation}

\bigskip
We have to $\Big(  F_{a,b}^l (\psi)  ,  d\epsilon_{it}  \Big)  $
\begin{align}
   & \, = \, 
\int_{\Gamma \backslash G }   F_{a,b}^l (\psi) \big(   z+\lambda j, K \big)  \cdot  E \big( z+ \lambda j,  it \big)
 \;  \sum_{\substack{ L \in  \mathbb{Z} \\ L \geq 0 }}  \;  \sum_{\substack{ k,m=-\rfrac{L}{2} \\ k,m \in \tfrac{1}{2} \Z \\ k \equiv m \equiv \rfrac{L}{2} \text{ mod } 1 }}^{\rfrac{L}{2}}   (-1)^{m-k} \cdot \mathcal{T}^{\, L}_{kk} (I) \cdot  {E}_{-m,-k}^{\rfrac{L}{2}} \big(z +\lambda j, -it\big)    \; \cdot  \; \overline{\mathcal{T}_{km}^L(K)} \, d g 
 &&  \nonumber   \\
   & \, = \, 
\int_{\Gamma \backslash G }  \bigg(  \frac{1}{2 \pi i} \, 
\sum_{\substack{ r=-l \\ r \in \tfrac{1}{2} \Z \\ r \equiv l \text{ mod } 1 }}^{l}    \overline{D^l_{ar} \big( \Phi (K)^{-1} \big)}  \int_{(3)}  H(s)  \cdot E_{rb}^l (z+ \lambda j ,s-1) \; ds \bigg) \;  E \big( z+ \lambda j, it \big)
 && \nonumber \\
   & \hspace{3cm} \cdot \; \sum_{\substack{ L \in  \mathbb{Z} \\ L \geq 0 }}   \;  \sum_{\substack{ k,m=-\rfrac{L}{2} \\ k,m \in \tfrac{1}{2} \Z \\ k \equiv m \equiv \rfrac{L}{2} \text{ mod } 1 }}^{\rfrac{L}{2}}   (-1)^{m-k} \cdot \mathcal{T}^{\, L}_{kk} (I) \cdot  {E}_{-m,-k}^{\rfrac{L}{2}} \big(z +\lambda j, -it\big)    \; \cdot  \; \overline{\mathcal{T}_{km}^L(K)} \, d g 
    \hspace{3cm} \textrm{por (\ref{IES0})}&&   \nonumber \\
   & \, = \,  \sum_{\substack{ r=-l \\ r \in \tfrac{1}{2} \Z \\ r \equiv l \text{ mod } 1 }}^{l}   \int_{(3)}  \frac{1}{2 \pi i} \, H(s) 
\int_{ \Gamma \backslash \mathbb{H}^3 }  E_{rb}^l (z+\lambda j,s-1) \cdot E \big( z+ \lambda j,  it \big) 
   \sum_{\substack{ L \in  \mathbb{Z} \\ L \geq 0 }}  \;   \sum_{\substack{ k,m=-\rfrac{L}{2}  \\ k,m \in \tfrac{1}{2} \Z \\  k \equiv m \equiv \rfrac{L}{2} \text{ mod } 1 }}^{\rfrac{L}{2}}   (-1)^{m-k} \cdot \mathcal{T}^{\, L}_{kk} (I) \cdot  {E}_{-m,-k}^{\rfrac{L}{2}} \big(z +\lambda j, -it\big) && \nonumber \\
   &  \hspace{4.5cm}  \bigg[ \int_{\SU(2) }  
 \overline{D^l_{ar} \big( \Phi (K)^{-1} \big)}  \cdot \overline{\mathcal{T}_{km}^L(K)} \; d k \bigg] 
d V  \, ds. && \label{InES1}
 \end{align} 

\bigskip
Working now with the integral in $\SU(2)$ of the equation (\ref{InES1}). Using the identities (\ref{A8}), (\ref{A11}), (\ref{T1}) y (\ref{A10}) it is seen that 
\begin{equation}\label{InE2}
\int_{\SU(2) }  
 \overline{D^l_{ar} \big( \Phi (K)^{-1} \big)}  \cdot \overline{\mathcal{T}_{km}^L(K)} \; d k \, = \,  \frac{1}{\sqrt{2} \pi \,\sqrt{2l+1}} \;
   \delta_{l,\rfrac{L}{2}} \, \delta_{r,m}  \, \delta_{a,k}. 
\end{equation}

 \bigskip 
 As a consequence, replacing (\ref{InE2}) in (\ref{InES1}) the inner product is:
 $$    \frac{1}{\sqrt{2} \pi \,\sqrt{2l+1}}  \int_{(3)}  \frac{1}{2 \pi i} \, H(s) 
\int_{ \Gamma \backslash \mathbb{H}^3 }  E \big( z+ \lambda j, it \big)  \sum_{\substack{ m=-l \\ m \in \tfrac{1}{2} \Z \\ m \equiv l \text{ mod } 1 }}^{l}    (-1)^{m-a} \cdot \mathcal{T}^{\, 2l}_{aa} (I) \cdot E_{mb}^l (z+\lambda j,s-1) \cdot
    {E}_{-m,-a}^{l} \big(z +\lambda j, -it\big)    
d V  \, ds$$ 
and developing  $E_{mb}^l (z+\lambda j,s-1)$,
 $$= \,  \frac{1}{2 \pi^2} \int_{(3)}  \frac{1}{2 \pi i} \, H(s) 
\int_{ \Gamma \backslash \mathbb{H}^3 }  E \big( z+ \lambda j,  it \big)  \sum_{\substack{ m=-l \\ m \in \tfrac{1}{2} \Z \\ m \equiv l \text{ mod } 1 }}^{l}   (-1)^{m-a}   \sum_{\sigma\,\in\,_{\scriptstyle\Gamma'_\infty}\backslash\Gamma}   \overline{ D_{mb}^l \big(  \Phi_{ T( \sigma n[z]a[\lambda]  )^{-1}} \big)} \hspace{4.3cm} $$
\begin{equation}\label{IES1}
\cdot \,  \Im \,  \sigma ( z+\lambda j)^{s} \cdot   {E}_{-m,-a}^{l} \big(z +\lambda j, -it\big)  \;  
d V  \, ds.
\end{equation}

\bigskip
As earlier, we make the variable change
$$ \sigma ( z + \lambda j ) \, = \, z'+\lambda' j \; \; \iff \; \;  z + \lambda j  \, = \, \sigma^{-1} ( z'+\lambda' j).$$
So that 
$$ n[z]a[\lambda] \, = \, \sigma^{-1} \, n[z']a[\lambda'] \, A^{-1}, $$
with $A:= T \big(  \sigma^{-1} \, n[z']a[\lambda'] \big)$. Using the invariance for the Eisenstein series we have that the integral in (\ref{IES1})  is
\begin{equation}\label{IES2}
 \int_{(3)}  \frac{1}{2 \pi i} \, H(s)   \sum_{\sigma\,\in\,_{\scriptstyle\Gamma'_\infty}\backslash\Gamma} 
\int_{\sigma( \Gamma \backslash \mathbb{H}^3) }  E \big(  z'+\lambda' j,  it \big)   \sum_{\substack{ m=-l \\ m \in \tfrac{1}{2} \Z \\ m \equiv l \text{ mod } 1 }}^{l}   (-1)^{m-a} \cdot    \overline{ D_{mb}^l \big(  \Phi ( A ) \big)} \; \lambda'^{s} \cdot     {E}_{-m,-a}^{l} \big( n[z']a[\lambda'] \, A^{-1}, -it\big)    
d V  \, ds.
\end{equation}

For (\ref{PropE})
\begin{equation}\label{IES3}
 {E}_{-m,-a}^l \big( n[z']a[\lambda'] \, A^{-1}, -it \big)   \, = \,    \sum_{\substack{ u=-l \\ u \in \tfrac{1}{2} \Z \\ u \equiv l \text{ mod } 1 }}^{l}   \overline{D^l_{-m,u} \big( \Phi (A) \big)}  \cdot E_{u,-a}^l (z' +\lambda' j, -it).
\end{equation}

\bigskip
Replacing (\ref{IES3}) in  (\ref{IES2})
 $$ \textbf{I}(t) \, := \,  \frac{1}{2 \pi^2}  \; \int_{(3)}  \frac{1}{2 \pi i} \, H(s)   \sum_{\sigma\,\in\,_{\scriptstyle\Gamma'_\infty}\backslash\Gamma} 
\int_{\sigma( \Gamma \backslash \mathbb{H}^3) }  E \big( z'+\lambda' j, it \big)   \sum_{\substack{ m=-l \\ m \in \tfrac{1}{2} \Z \\ m \equiv l \text{ mod } 1 }}^{l}  (-1)^{m-a} \cdot    \overline{ D_{mb}^l \big(  \Phi ( A) \big)} \cdot \lambda'^{s}$$ $$  \cdot  \; \sum_{\substack{ u=-l \\ u \in \tfrac{1}{2} \Z \\ u \equiv l \text{ mod } 1 }}^{l}     \overline{D^l_{-m,u} \big( \Phi ( A) \big)}  \cdot E_{u,-a}^l (z' +\lambda' j, -it)    
d V  \, ds$$
 $$= \; \frac{1}{2 \pi^2}  \; \int_{(3)}  \frac{1}{2 \pi i} \, H(s)   \sum_{\sigma\,\in\,_{\scriptstyle\Gamma'_\infty}\backslash\Gamma} 
\int_{\sigma( \Gamma \backslash \mathbb{H}^3) } \lambda'^{s} \cdot E \big(  z'+\lambda' j, it \big)  \sum_{\substack{ u=-l \\ u \in \tfrac{1}{2} \Z \\ u \equiv l \text{ mod } 1 }}^{l}     E_{u,-a}^l (z' +\lambda' j, -it) $$
\begin{equation}\label{IES4}
\hspace{1cm} \bigg[  \sum_{\substack{ m=-l \\ m \in \tfrac{1}{2} \Z \\ m \equiv l \text{ mod } 1 }}^{l}  (-1)^{m-a} \cdot    \overline{ D_{mb}^l \big(  \Phi (A) \big)} \;  \cdot     \overline{D^l_{-m,u} \big( \Phi (A) \big)} \,  \bigg]  \,  d V  \, ds.
\end{equation}

\bigskip
Now we take care of the sum in (\ref{IES4}) between brackets, for (\ref{A8}) and (\ref{WI10})
\begin{equation}\label{IES5} 
\sum_{\substack{ m=-l \\ m \in \tfrac{1}{2} \Z \\ m \equiv l \text{ mod } 1 }}^{l}  (-1)^{m-a} \cdot \overline{ D_{mb}^l \big(  \Phi ( A) \big)}  \cdot     \overline{D^l_{-m,u} \big( \Phi ( A ) \big)}   \, = \, (-1)^{a+b}   \cdot  \delta_{u,-b}.
\end{equation}

\bigskip
Then, replacing (\ref{IES5}) in (\ref{IES4}) 
 $$\textbf{I}(t) \, = \,   (-1)^{a+b}  \, \frac{1}{2 \pi^2}   \; \int_{(3)}  \frac{1}{2 \pi i} \, H(s)   \sum_{\sigma\,\in\,_{\scriptstyle\Gamma'_\infty}\backslash\Gamma} 
\int_{\sigma( \Gamma \backslash \mathbb{H}^3) } \lambda^{s} \cdot E \big(  z+\lambda j, it \big) \cdot   E_{-b,-a}^l (z +\lambda j, -it) \,     
d V  \, ds$$
\begin{equation}\label{IES6}
= \, (-1)^{a+b} \, \frac{1}{2 \pi^2}   \; \int_{(3)}  \frac{1}{2 \pi i} \, H(s)   \int_0^\infty   \lambda^{s} \bigg[
\int_{[0,1]^2} E \big(  z+\lambda j, it \big) \cdot   E_{-b,-a}^l (z +\lambda j,-it)      
\; dx dy \bigg] \frac{d\lambda}{\lambda^3} \, ds.
\end{equation}

\bigskip
The Fourier expansion of the classic Eisenstein series associated to $\Gamma$ and evaluated in $s=it$ (with $t \neq 0$) is given by:
$$    E (z+\lambda j,it) \, = \, [\Gamma_{\infty} : \Gamma'_{\infty} ]  \; \lambda^{1+it} \;  - \;   \frac{i}{t} \;
\frac{ \zeta_K ( it )}{ \zeta_K ( 1+it )}  \; \lambda^{1-it} \hspace{5cm}$$
\begin{equation}\label{IES7}
\hspace{2cm} + \; \frac{(2\pi)^{1+it} \, \lambda}{  \Gamma(1+it) } \,      \sum_{ 0 \neq \alpha \in \Lambda'} 
\mathcal{D}_{0} (-\alpha; it) \cdot \abs{ \alpha }^{it}  \cdot  K_{it} \big( 4 \pi \abs{\alpha} \lambda \big) \cdot e^{-4 \pi i \, \Re(\alpha z )}.
\end{equation}

\bigskip
The Fourier expansion of the Eisenstein series generalized and evaluated in $s=-it$ (with $t \neq 0$) is:
$$  E_{-b,-a}^l (z+\lambda j,-it) \, = \, [\Gamma_{\infty} : \Gamma'_{\infty} ]   \; \delta_{b,a} \, B_{b,a}^l \, \lambda^{1-it} \; + \;  (-1)^{-a+\abs{a} } \pi \, \frac{\Gamma(1+l+it) \, \Gamma(\abs{a}-it)}{\Gamma(1+l-it) \, \Gamma(1+\abs{a}+it)} \,
\frac{L\big(-it,\chi_{-2a} \big)}{L( 1-it,\chi_{-2a})} \, \delta_{b,-a} \, B_{b,a}^l \, \lambda^{1+it} $$
$$ \, + \,  (-1)^{l-a}  \, i^{b+a} \, (2 \pi)^{-it}  \, B_{b,a}^l   \Large \sum_{ 0 \neq w \in \Lambda'} 
\mathcal{D}_{-a} (-w; -it) \cdot  \abs{ w }^{-it-1} \cdot e^{-4 \pi i \, \Re(w z )}    \Big( \frac{ w }{\abs{w}}  \Big)^{a+b} \hspace{1cm} $$
 \begin{equation}\label{IES8}
 \cdot  \,  \sum_{ u=0 }^{ l- \frac{1}{2}( \abs{a+b}+
\abs{a-b})} (-1)^u \;  \xi_{a}^l (b,u)  \, \frac{\big(2 \pi \abs{w} \lambda \big)^{1+l-u}}{\Gamma(1+l-u-it)} \cdot  K_{l-\abs{a+b}-u-it}   \big( 4 \pi \abs{w} \lambda \big).  
 \end{equation}

%Then, 
%$$ \int_{P_\Lambda} E \big(  z+\lambda j, it \big) \cdot   E_{-b,-a}^l (z +\lambda j,-it)  \; dx dy  =  [\Gamma_{\infty} : \Gamma'_{\infty} ]^2   \; \delta_{b,a} \, B_{b,a}^l  \, \abs{\Lambda_D} \; \lambda^{2} $$
%$$ + \;  (-1)^{l-a}  \, i^{b+a} \, \frac{(2\pi)^{2+l} }{ \abs{ \Lambda_D} \cdot \Gamma(1+it) } \,   B_{b,a}^l \,
% \Bigg[ \sum_{ 0 \neq w \in \Lambda'}  \mathcal{D}_{-a} (-w; -it) \cdot  
%\mathcal{D}_{0} (w; it) \cdot \abs{ w }^{l-u}  \cdot   \bigg( \frac{ w }{\abs{w}}  \bigg)^{a+b} \Bigg]$$
%$$ \cdot \;  \sum_{ u=0 }^{ l- \frac{1}{2}( \abs{a+b}+\abs{a-b})} (-1)^u 
%\;  \xi_{a}^l (b,u) \; 
% \cdot  \,   \, \frac{\big(2 \pi  \big)^{-u}  }{\Gamma(1+l-u-it)} \cdot  K_{l-\abs{a+b}-u-it}   \big( 4 \pi \abs{w} \lambda \big) 
% \cdot  K_{it} \big( 4 \pi \abs{w} \lambda \big) \cdot  \lambda^{2+l-u} $$

\bigskip
From (\ref{IES6}), (\ref{IES7}) and (\ref{IES8}) we have that
$$ \textbf{I}(t) \, = \, (-1)^{a+b} \, \delta_{b,a} \, B_{b,a}^l  \, [\Gamma_{\infty} : \Gamma'_{\infty} ]^2    \,  \frac{1}{2 \pi^2}
\; \int_{(3)}  \frac{1}{2 \pi i} \, H(s)   \int_0^\infty   \lambda^{s-1} 
\; d \lambda  \, ds \; + \;  \Big( 
\text{rapidly decreasing in } t \Big) $$
% acá voy
$$ + \,  (-1)^{l+b}  \, i^{b+a} \, \frac{2 \, (2\pi)^{l} }{  \Gamma(1+it) } \,   B_{b,a}^l  \;  \sum_{ u=0 }^{ l- \frac{1}{2}( \abs{a+b}+\abs{a-b})}  
\;  \xi_{a}^l (b,u) \, \frac{(-1)^u \, (2 \pi)^{-u}  }{\Gamma(1+l-u-it)} \hspace{1cm}$$    
$$  \int_{(3)}  \frac{1}{2 \pi i} \, H(s) 
   \, \Bigg[ \sum_{ 0 \neq w \in \Lambda'}  \mathcal{D}_{-a} (-w; -it) \cdot  
\mathcal{D}_{0} (w; it) \cdot \abs{ w }^{l-u}  \cdot   \bigg( \frac{ w }{\abs{w}}  \bigg)^{a+b} \Bigg] $$ 
$$ \bigg[ \int_0^\infty  \lambda^{-1+s+l-u} \cdot  K_{l-\abs{a+b}-u-it}   \big( 4 \pi \abs{w} \lambda \big)  \cdot  K_{it} \big( 4 \pi \abs{w} \lambda \big) \; d\lambda \bigg] \, ds. $$

\bigskip
And making the variable change $ x =  4 \pi \abs{w} \lambda$ in the last integral,
$$ \textbf{I}(t) \, = \, (-1)^{a+b} \, \delta_{b,a} \, B_{b,a}^l  \, [\Gamma_{\infty} : \Gamma'_{\infty} ]^2   \;   \frac{1}{2 \pi^2}
\; \int_{(3)}  \frac{1}{2 \pi i} \, H(s)   \int_0^\infty   \lambda^{s-1} 
\; d \lambda  \, ds \; + \;  \Big( 
\text{rapidly decreasing in } t \Big) $$

$$ + \,  (-1)^{l+b}  \, i^{b+a} \, \frac{2^{1-l}}{ \Gamma(1+it) } \,   B_{b,a}^l \;  \sum_{ u=0 }^{ l- \frac{1}{2}( \abs{a+b}+\abs{a-b})}  
\;  \xi_{a}^l (b,u) \, 
\frac{(-2)^u }{\Gamma(1+l-u-it)} $$
$$   \int_{(3)} \, (4 \pi)^{-s} \, \frac{1}{2 \pi i} \, H(s) \, \Bigg[ \sum_{ 0 \neq w \in \Lambda'}  \mathcal{D}_{-a} (-w; -it) \cdot  
\mathcal{D}_{0} (w; it) \cdot \abs{ w }^{-s}  \cdot   \bigg( \frac{ w }{\abs{w}}  \bigg)^{a+b} \Bigg]$$
\begin{equation}\label{IES10}
  \bigg[  \int_0^\infty  x^{-1+s+l-u}  \cdot     
   K_{l-\abs{a+b}-u-it}   (x) \cdot  K_{it} (x) \; dx \bigg] \, ds.
\end{equation}

%\bigskip
% Using the Mellin inversion formula (\ref{IM}) we conclude that in the first term in the last expresion the integral is given by:
%\begin{equation}\label{IES11}
%\int_{(3)}  \frac{1}{2 \pi i} \, H(s)   \int_0^\infty   \lambda^{s-1} \, d\lambda \, ds \, = \,  \int_0^\infty  \bigg[ 
% \frac{1}{2 \pi i}  \int_{(3)} H(s) \, \lambda^{s} \, ds \bigg] \, \frac{d \lambda}{\lambda}  \, = \,  \int_0^\infty  \psi(\lambda) \, \frac{d \lambda}{\lambda}   \, = \,  H(0) .
%\end{equation}

\bigskip
The integral at the end of the expresion (\ref{IES10}) can be found in Gradshteyn and Ryzhik   \cite{GR}  page 692.
$$ \int_0^\infty x^{\,-z} \cdot K_{\mu} ( x )  \cdot K_{\nu}   ( x ) \;  dx   \, = \,  \frac{2^{-2-z}}{\Gamma(1-z)} \Gamma \big( \tfrac{1-z+\mu+\nu}{2}  \big) \Gamma \big( \tfrac{1-z-\mu+\nu}{2}  \big) \Gamma \big( \tfrac{1-z+\mu-\nu}{2}  \big) \Gamma \big( \tfrac{1-z-\mu-\nu}{2}  \big) \cdot F \big( \tfrac{1-z+\mu+\nu}{2} , \tfrac{1-z-\mu+\nu}{2} ;1-z ; 0 \big), $$
for $\Re (z) <  1- \abs{\Re(\mu)} - \abs{\Re(\nu)}$. Particularly, 
$$  \int_0^\infty  x^{s-1+l-u}
\cdot  K_{l-\abs{a+b}-u-it} \, (x) \cdot K_{it} \, (x) \, d x \, $$
\begin{equation}\label{IE13}
 = \;  \frac{2^{s+l-3-u}}{\Gamma ( s+l-u) } \,
  \Gamma \big( \tfrac{s}{2}+l-u-\tfrac{1}{2} \abs{a+b}  \big) \, 
  \Gamma \big( \tfrac{s}{2}+\tfrac{1}{2} \abs{a+b}+it \big) \, 
  \Gamma \big( \tfrac{s}{2}+l-u-\tfrac{1}{2} \abs{a+b} -it \big) \,
    \Gamma \big( \tfrac{s}{2}  +  \tfrac{1}{2} \abs{a+b} \big)
\end{equation}
for $-l-\Re (s) + u <   - \abs{l  - \abs{a+b}-u}$.

\bigskip
As $ \Lambda' \, = \, \tfrac{1}{2} \Lambda$ the summation of the expression (\ref{IES10})  is the following:
\begin{equation} \label{IE18}
  S\, = \, 2^s \, \sum_{ 0 \neq w \in Z[i]} 
\mathcal{D}_{-a} (-\tfrac{w}{2} ; -it) \cdot 
 \mathcal{D}_0 (\tfrac{w}{2} ;  it) \cdot  \abs{w}^{-s} \, 
   \bigg( \frac{ w }{\abs{w}}  \bigg)^{a+b}.
\end{equation}

\bigskip
We will use the following identity Ramanujan type. For every $\alpha_1,\alpha_2,\alpha_3 \in \Z$, $\mu, \nu \in \C$, in the convergence region the following formula is valid:
$$  \sum_{0 \neq  w \in Z[i]}  \bigg( \frac{ w }{\abs{w}}  \bigg)^{4\alpha_1} \frac{1}{  \abs{w}^{2s} }    \cdot \sigma_\mu (w,\alpha_2) \cdot \sigma_\nu (w,\alpha_3)   \hspace{10cm} $$ 
\begin{equation}\label{IES115}
\hspace{3cm} = \; \frac{1}{16} \cdot \frac{ L(s, \chi_{4\alpha_1} ) \cdot L( s -\mu, \chi_{4\alpha_1+4\alpha_2} ) \cdot L( s-\nu, \chi_{4\alpha_1+4\alpha_3} ) \cdot L(s-\mu-\nu, \chi_{4\alpha_1+4\alpha_2+4\alpha_3} ) }{L( 2s-\mu-\nu, \chi_{8\alpha_1+4\alpha_2+4\alpha_3} )}.
\end{equation}

%$$\sum_{0 \neq  w \in \mathcal{O}}  \bigg( \frac{ w }{\abs{w}}  \bigg)^{4\alpha_1} \frac{1}{  \abs{w}^{2s} }    \cdot \sigma_\mu (w,\alpha_2) \cdot \sigma_\nu (w,\alpha_3)    \, = \,  $$ $$ 
%4 \cdot  \frac{ \zeta_K(s, \alpha_1 ) \cdot \zeta_K ( s-\mu, \alpha_1+\alpha_2 ) \cdot \zeta_K (s-\nu, \alpha_1+\alpha_3 ) \cdot \zeta_K (s- \mu- \nu, \alpha_1+\alpha_2+\alpha_3) }{\zeta_K (2s-\mu-\nu,2\alpha_1+\alpha_2+\alpha_3)}.$$

%\bigskip
%Equivalent to,
%$$ \sum_{0 \neq  w \in \mathcal{O}}  \bigg( \frac{ w }{\abs{w}}  \bigg)^{2\alpha_1} \frac{1}{  \abs{w}^{2s} }    \cdot \sigma_\mu (w,\alpha_2) \cdot \sigma_\nu (w,\alpha_3)    \, = \,  $$ $$ 
%4 \cdot  \frac{ \zeta_K(s, \tfrac{\alpha_1}{2} ) \cdot \zeta_K ( s-\mu, \tfrac{\alpha_1}{2}+\alpha_2 ) \cdot \zeta_K (s-\nu, \tfrac{\alpha_1}{2}+\alpha_3 ) \cdot \zeta_K (s- \mu- \nu, \tfrac{\alpha_1}{2}+\alpha_2+\alpha_3) }{\zeta_K (2s-\mu-\nu,\alpha_1+\alpha_2+\alpha_3)}. $$

\bigskip
Replacing (\ref{IE20}) in (\ref{IE18}) and using the identity $\sigma_s (-w, p) = \sigma_s (w, p)$ for each $p \in \Z$ we have that 
$$ S \, = \,  2^s \,  \frac{16}{L(1+it,\chi_0)}
\cdot  \frac{16}{L(1-it,\chi_{-2a})}
 \sum_{0 \neq  w \in Z[i]}  \bigg( \frac{  w  }{\abs{w}}  \bigg)^{a+b}    \frac{1}{  \abs{w}^{s} } \cdot \sigma_{-it} (w,0) \cdot \sigma_{it} (w,-\tfrac{a}{2}).$$

\bigskip 
Then, for (\ref{IES115})
 \begin{equation}\label{IE22}
S \, = \,  2^s \, \frac{4}{L(1+it,\chi_0)}
\cdot  \frac{4}{L(1-it,\chi_{-2a})} \cdot
 \frac{ L(\tfrac{s}{2}, \chi_{a+b} ) \cdot  L(\tfrac{s}{2}  +it, \chi_{a+b} ) \cdot  L(\tfrac{s}{2}-it, \chi_{-a+b} ) \cdot L(\tfrac{s}{2}, \chi_{-a+b} ) }{L( s, \chi_{2b} )}. 
 \end{equation}

\bigskip
Finally, replacing  (\ref{IE13}) and (\ref{IE22}) in (\ref{IES10}) we conclude that 
\begin{equation}\label{IES25}
\Big(  F_{a,b}^l (\psi)  ,  d\epsilon_{it}  \Big) \, = \,  F_1 (t ) \, + \, F_2(t)
\end{equation}
where
$$
   F_1 (t ) \, = \,  (-1)^{a+b} \, \delta_{b,a} \, B_{b,a}^l  \, [\Gamma_{\infty} : \Gamma'_{\infty} ]^2   \,  \frac{1}{2 \pi^2}
\; \int_{(3)}  \frac{1}{2 \pi i} \, H(s)   \int_0^\infty   \lambda^{s-1} 
\; d \lambda  \, ds \, + \,  \Big( 
\text{rapidly decreasing in } t \Big), $$

$$ F_2 (t) \, = \,  (-1)^{l+b}  \, i^{b+a} \,  B_{b,a}^l \, \frac{ 4 }{ \Gamma(1+it) } \cdot  \frac{1}{L(1+it,\chi_0) \cdot L(1-it,\chi_{-2a})} $$
\begin{equation}\label{F2}
\sum_{ u=0 }^{ l- \frac{1}{2}( \abs{a+b}+\abs{a-b})}  
\frac{(-1)^u \; \xi_{a}^l (b,u) }{\Gamma(1+l-u-it)}  \;  \int_{(3)}  \frac{1}{2 \pi i} \,   B(s) \; ds, 
\end{equation}
and 
$$   B(s) \, = \,  \frac{   H(s) \cdot L(\tfrac{s}{2}, \chi_{a+b} ) \cdot  L(\tfrac{s}{2}  +it, \chi_{a+b} ) \cdot  L(\tfrac{s}{2}-it, \chi_{-a+b} ) \cdot L(\tfrac{s}{2}, \chi_{-a+b} ) }{\pi^s \; \Gamma ( s+l-u) \cdot L( s, \chi_{2b} )} $$
$$ \cdot \, 
  \Gamma \big( \tfrac{s}{2}+l-u-\tfrac{1}{2} \abs{a+b}  \big) \cdot 
  \Gamma \big( \tfrac{s}{2}+\tfrac{1}{2} \abs{a+b}+it \big) \cdot 
  \Gamma \big( \tfrac{s}{2}+l-u-\tfrac{1}{2} \abs{a+b} -it \big) \cdot
    \Gamma \big( \tfrac{s}{2}  +  \tfrac{1}{2} \abs{a+b} \big).
$$

\bigskip
Applying the residue theorem and rearranging terms we have that
$$ \Big(  F_{a,b}^l (\psi)  ,  d\epsilon_{it}  \Big) \, = \, O(1) \, + \, 
4  (-1)^{l+b}  \, i^{b+a} \,  B_{b,a}^l  $$ 
$$ \sum_{ u=0 }^{ l- \frac{1}{2}( \abs{a+b}+\abs{a-b})}   \frac{(-1)^u \; \xi_{a}^l (b,u) }{  \Gamma(1+it) \cdot \Gamma(1+l-u-it) \cdot L(1+it,\chi_0) \cdot L(1-it,\chi_{-2a})} \,  \int_{(1)}  \frac{1}{2 \pi i} \,   B(s) \; ds $$
\begin{equation}\label{Cuspidal1}
+ \,  4(-1)^{l+b}  \, i^{b+a} \,  B_{b,a}^l \, 
\sum_{ u=0 }^{ l- \frac{1}{2}( \abs{a+b}+\abs{a-b})}  
  \frac{(-1)^u \; \xi_{a}^l (b,u) }{  \Gamma(1+it) \cdot \Gamma(1+l-u-it) \cdot L(1+it,\chi_0) \cdot L(1-it,\chi_{-2a})}  \; \text{Res}_{s=2} B(s).
\end{equation}

\bigskip
Now we take care of estimate the second adding in (\ref{Cuspidal1}), for that we consider  
\begin{equation}\label{Cuspidal2}
 A(t) \, := \, \frac{1}{  \Gamma(1+it) \cdot \Gamma(1+l-u-it) \cdot L(1+it,\chi_0) \cdot L(1-it,\chi_{-2a})} \,  \int_{(1)}  \frac{1}{2 \pi i} \,   B(s) \; ds.
\end{equation}

\bigskip
We make $s=1+i\tau$ in the formula for $B(s)$, then 
$$ A(t) \, = \,  \frac{1}{ \Gamma(1+it) \cdot \Gamma(1+l-u-it) \cdot L(1+it,\chi_0) \cdot L(1-it,\chi_{-2a})} $$
$$ \cdot \,  \int_{-\infty}^{\infty}
 \frac{  H(1+i\tau) \cdot L(\tfrac{1}{2}+i\tfrac{\tau}{2}, \chi_{a+b} ) \cdot  L(\tfrac{1}{2} +it+i\tfrac{\tau}{2} , \chi_{a+b} ) \cdot  L(\tfrac{1}{2}-it+i\tfrac{\tau}{2}, \chi_{-a+b} ) \cdot L(\tfrac{1}{2}+i\tfrac{\tau}{2}, \chi_{-a+b} ) }{2 \pi^{2+i\tau} \, \Gamma ( 1+l-u+i\tau) \cdot L( 1+ i\tau, \chi_{2b} )} $$
 \begin{equation}\label{cuspidal3}
 \cdot \; \Gamma \big( l+\tfrac{1}{2}-u-\tfrac{1}{2} \abs{a+b} +i\tfrac{\tau}{2} \big) \cdot 
  \Gamma \big( \tfrac{1}{2}+\tfrac{1}{2} \abs{a+b}+it + i\tfrac{\tau}{2} \big) \cdot 
  \Gamma \big(\tfrac{1}{2}+l-u-\tfrac{1}{2} \abs{a+b} -it +i\tfrac{\tau}{2} \big) \cdot
    \Gamma \big(  \tfrac{1}{2} +  \tfrac{1}{2} \abs{a+b}+i\tfrac{\tau}{2}  \big) \, d \tau.
 \end{equation}

%\bigskip
%We remember the Stirling's Formula $ \abs{\Gamma (\sigma \pm ir)} \, \sim  \, \sqrt{2 \pi}\, e^{-\tfrac{\pi}{2}\abs{ r }} \cdot \abs{r}^{\, \sigma-\tfrac{1}{2}}$ as $r \longrightarrow \infty$. 

\bigskip
It is  known that 
\begin{equation}\label{Bo1}
\ln^{-2} \, \abs{t} \ll \abs{ \zeta_K (1+it) } \ll \ln^2 \, \abs{t}.
\end{equation}

\bigskip % FAULT
Using results in \cite{Landau} and \cite{Ti} the above can be generalized as follows:
\begin{equation}\label{Bo4}
\ln^{-2} \, \abs{t} \ll \abs{ L(1+it, \chi_p ) } \ll \ln^2 \, \abs{t}, \; \; \forall p \in 4 \Z.
\end{equation}

\bigskip
The bound for Heath-Brown \cite{HB} says that for every $\epsilon > 0$ 
$$ \zeta_K (\tfrac{1}{2} + it) \ll_K t^{\rfrac{1}{3}+ \epsilon}, \; \; t \geq  1.  $$
The version that generalizes the previous that we will use is due to Kaufman \cite{Kau} and Sohne \cite{Soh} 
\begin{equation}\label{Bo2}
L(\tfrac{1}{2}+it, \chi_p ) \ll \big( 1+\abs{t} \big)^{\rfrac{1}{3}+ \epsilon}, \; \; \forall p \in 4 \Z, \; \; \forall \epsilon > 0.
\end{equation}

\bigskip
Stirling exponential asymptotics for $B(1+i \tau)$ as a function of $t$ gives
$$
\displaystyle \frac{ e^{-\tfrac{\pi}{2} \,\abs{\tfrac{\tau}{2}} } \cdot 
e^{-\tfrac{\pi}{2} \, \abs{t+\tfrac{\tau}{2} }} \cdot
e^{-\tfrac{\pi}{2} \, \abs{t-\tfrac{\tau}{2} }} \cdot
e^{-\tfrac{\pi}{2} \,\abs{\tfrac{\tau}{2}} } }{ e^{-\tfrac{\pi}{2} \, \abs{\tau} } }
\; \leq \; e^{- \pi t} ,$$
which cancels with the exponential growth of denominator in (\ref{Cuspidal2}).

\bigskip
Using (\ref{Bo1}), (\ref{Bo4}), (\ref{Bo2}), Stirling's formula and the rapid decay of $H(1+i\tau)$ we are reduced to estimate in $t$ the integral in (\ref{cuspidal3}), this is 
$$
\frac{ \ln^2 \, \abs{t} \cdot \ln^2 \, \abs{t} }{ \abs{t}^{1+l-u} }
\,  \int_{-\infty}^{\infty} H(1+ i \tau) \cdot 
\big( 1 + \abs{t+\tfrac{\tau}{2}} \big)^{\rfrac{1}{3}+\epsilon}
   \cdot \big( 1 + \abs{t-\tfrac{\tau}{2}} \big)^{\rfrac{1}{3}+\epsilon} \, d \tau
   \, = \, O\big( t^{-\rfrac{1}{3}-(l-u)+\epsilon} \big).
$$

\bigskip
This concludes that 
\begin{equation}\label{Bo5}
 A(t) \, = \, O\big( t^{-\rfrac{1}{3}-(l-u)+\epsilon} \big).
\end{equation}

\bigskip
Summarizing, from the formulas (\ref{Cuspidal1}) and (\ref{Bo5}) we have that 
$$ \Big(  F_{a,b}^l (\psi)  ,  d\epsilon_{it}  \Big) \, = \, O(1) \; + \; 4 (-1)^{l+b}  \, i^{b+a} \,  B_{b,a}^l  $$
\begin{equation}\label{F3}
\cdot \, \sum_{ u=0 }^{ l- \frac{1}{2}( \abs{a+b}+\abs{a-b})}  
  \frac{(-1)^u \; \xi_{a}^l (b,u) }{  \Gamma(1+it) \cdot \Gamma(1+l-u-it) \cdot L(1+it,\chi_0) \cdot L(1-it,\chi_{-2a})}  \; \text{Res}_{s=2} B(s).
\end{equation}

\bigskip
\bigskip
For the previous formula we require to find the residue of  $B(s)$ in $s=2$. This will be made for cases. First we consider  $a=b=0$ and $t \neq 0$. In this conditions we have that 
$$ B(s) \, = \,  \frac{  H(s) \cdot  \zeta_K(\tfrac{s}{2} )^2 \cdot  \zeta_K (\tfrac{s}{2}  +it) \cdot  \zeta_K (\tfrac{s}{2}-it) \cdot \; \Gamma \big( \tfrac{s}{2} +l-u \big) \cdot   \Gamma \big( \tfrac{s}{2}+it \big) \cdot 
  \Gamma \big( \tfrac{s}{2}+l-u-it \big) \cdot  \Gamma \big( \tfrac{s}{2}  \big) }{\pi^{s} \; \Gamma (s+l-u) \cdot \zeta_K ( s )}. $$

\bigskip
Write  $B(s) = \zeta_K \big(\frac{s}{2} \big)^2 \cdot G(s)$, with $G(s)$ holomorphic in $s=2$. Put
$$  \zeta_K \big( \tfrac{s}{2} \big) \, = \, \frac{A_{-1}}{s-2} + A_0 + O(s-2) \; \; \text{ as } s \longrightarrow 2. $$
Then $$ B(s) \, = \, \Big( \frac{A_{-1}}{s-2} + A_0 + O(s-2) \Big)^2 \Big( G(2)+G'(2)(s-2)+O(s-2)^2  \Big). $$
The residue is
\begin{equation}\label{IES27}
\text{Res}_{s=2} \, B(s) \, = \, G(2) \, A_{-1} \Big( 2 A_0 + A_{-1} \frac{G'}{G}(2) \Big).
\end{equation}

\bigskip
As $$ G(s) \, = \,   \frac{  H(s) \cdot   \zeta_K (\tfrac{s}{2}  +it) \cdot  \zeta_K (\tfrac{s}{2}-it) \cdot \; \Gamma \big( \tfrac{s}{2} +l-u \big) \cdot   \Gamma \big( \tfrac{s}{2}+it \big) \cdot 
  \Gamma \big( \tfrac{s}{2}+l-u-it \big) \cdot  \Gamma \big( \tfrac{s}{2}  \big) }{\pi^{s} \; \Gamma (s+l-u) \cdot \zeta_K ( s )} $$
it has that
\begin{equation}\label{IES28}
G(2) \, = \,   \frac{  H(2) \cdot   \zeta_K (1+it) \cdot  \zeta_K (1-it) \cdot  \Gamma ( 1+it ) \cdot \Gamma ( 1+l-u-it ) }{\pi^2 \, (1+l-u) \cdot \zeta_K ( 2 )}.
\end{equation}

\bigskip
On the other hand,
\begin{equation}\label{IES30}
\frac{G'}{G}(2) \, = \,  \frac{H'}{H}(2) + \frac{ \zeta_K'(1+it)}{2 \, \zeta_K(1+it)} +  \frac{ \zeta_K'(1-it)}{2 \, \zeta_K(1-it)}  +  \frac{ \Gamma' \big( 1+it \big) }{ 2 \, \Gamma \big( 1+it \big)}  +  \frac{ \Gamma' \big(1+l-u -it \big) }{ 2 \, \Gamma \big(1+ l-u -it \big) } + C,
\end{equation}
where $C$ is a constant that does not depend on $t$.

\bigskip
The bound to Dirichlet L-functions by Landau says
\begin{equation}\label{IES31}
 \frac{\zeta_K' (1+it)}{\zeta_K (1+it)}  \ll_K  \frac{\ln \, \abs{t}}{\ln \ln \, \abs{t}}
\end{equation}
as $ t \longrightarrow \infty$. %Verify with $\abs{t}$. falta 

\bigskip
It is known that (see Laaksonen \cite{Niko} formula (A.8) page 155)
\begin{equation}\label{IES32}
\frac{\Gamma' }{\Gamma} (1+it) \, = \, \ln \, \abs{t} \; + \; O(1).
\end{equation}

\bigskip
Moreover
\begin{equation}\label{IES33}
 A_{-1} \, = \,  \text{Res}_{s=1} \, \zeta_K(s) \, = \, \frac{\pi}{4}. 
\end{equation}

\bigskip
For the lemma (\ref{Ap2}) if $l-u \geq 1$
\begin{equation}\label{IES34.5}
 \frac{ \Gamma' \big(1+l-u -it \big) }{ 2 \, \Gamma \big( 1+l-u -it \big) } \, = \, \frac{1}{2} \sum_{k=0}^{l-u-1} \frac{1}{k+1-it} \; + \; \frac{\Gamma'(1-it)}{2 \, \Gamma(1-it)} \, = \, O(1) \, + \, \frac{1}{2} \ln t
\end{equation}
as $t \longrightarrow \infty$.

\bigskip 
Replacing (\ref{IES31}), (\ref{IES32}) and (\ref{IES34.5}) in (\ref{IES30})  we have
\begin{equation}\label{IES34.6}
\frac{G'}{G}(2) \, = \, O(1) \, + \, O \Big( \frac{\ln t}{\ln \ln t} \Big) \, + \,   \ln t.
\end{equation}
 
\bigskip
For (\ref{IES27}) and (\ref{IES34.6}),
\begin{equation}\label{ESI10}
\text{Res}_{s=2} \, B(s) \, = \, G(2) \, \Big( O(1) \, + \,  A_{-1}^2 \, \ln t \, + \,  O \Big( \frac{\ln t}{\ln \ln t} \Big)  \Big).
\end{equation}

From  (\ref{ESI10}) and (\ref{IES28})  the formula  (\ref{F3}) is transformed as follows
%$$ \Big(  F_{a,b}^l (\psi)  ,  d\epsilon_{it}  \Big) \, = \, O(1) \; + \;  (-1)^{l}   \, \frac{ 4 }{ \abs{ \Lambda_D} } \, $$
%$$ \cdot \, \sum_{ u=0 }^{ l }  
 % \frac{(-1)^u \; \xi_{0}^l (0,u) }{  \Gamma(1+it) \cdot \Gamma(1+l-u-it) \cdot L(1+it,%\chi_0) \cdot L(1-it,\chi_{0})}  \, G(2) \, \Big( O(1) \, + \,  A_{-1}^2 \, \ln t \, + \,  O \Big( \frac{\ln t}{\ln \ln t} \Big)  \Big) $$

%$$  = \; O(1) \; + \;  (-1)^{l}  \,  \frac{ 4 }{ \abs{ \Lambda_D} }  \, \sum_{ u=0 }^{ l }  
%  (-1)^u \; \xi_{0}^l (0,u) \, \frac{ H(2)}{\pi^2 (1+l-u) \cdot \zeta_K(2)} \cdot \bigg[ O(1) \, + \,  A_{-1}^2 \, \ln t \, + \,  O \Big( \frac{\ln t}{\ln \ln t} \Big)  \bigg] $$

\begin{equation}\label{IES98}
\Big(  F_{0,0}^l (\psi)  ,  d\epsilon_{it}  \Big) \, = \, O(1) \; + \;  \sum_{ u=0 }^{ l }     \frac{ (-1)^u \; \xi_{0}^l (0,u) }{ 1+l-u}  \bigg[ O \Big( \frac{\ln t}{\ln \ln t} \Big) \; + \;  (-1)^{l}  \,  \frac{H(2)}{4 \, \zeta_K(2)} \,  \ln t \bigg].
\end{equation}

\bigskip
In particular, 
\begin{equation}\label{IES99}
\Big(  F_{0,0}^0 (\psi)  ,  d\epsilon_{it}  \Big) \, = \, O(1) \; + \;  O \Big( \frac{\ln t}{\ln \ln t} \Big) \; + \;    \frac{H(2)}{4 \, \zeta_K(2)} \,  \ln t.
\end{equation}

\bigskip
For $l \geq 1 $ for the lemma (\ref{SumaEs0}) we know that the sum in $u$ in the equation (\ref{IES98}) is 0. Therefore
\begin{equation}\label{IES100}
\Big(  F_{0,0}^l (\psi)  ,  d\epsilon_{it}  \Big) \, = \, O(1).
\end{equation}

% ==========================================
\bigskip
\bigskip
We consider now the case  $a=b \neq 0$ and $t \neq 0$. Then, 
$$  B(s) \, = \, \frac{ H(s) \cdot L(\tfrac{s}{2}, \chi_{2a} ) \cdot  L(\tfrac{s}{2} +it, \chi_{2a} ) \cdot  \zeta_K \big( \tfrac{s}{2}-it  \big) \cdot \zeta_K \big(\tfrac{s}{2} \big)  }{ \pi^{s} \; \Gamma ( s+l-u) \cdot L( s, \chi_{2a} )} \hspace{7cm} $$
$$ \hspace{3cm} \cdot \,  \Gamma \big( \tfrac{s}{2}+l-u- \abs{a}  \big) \cdot \Gamma \big( \tfrac{s}{2}+ \abs{a}+it \big) \cdot 
  \Gamma \big( \tfrac{s}{2}+l-u- \abs{a} -it \big) \cdot
    \Gamma \big( \tfrac{s}{2}  +  \abs{a} \big). $$

\bigskip  
It is clear then that  $\infty$ is a simple pole of $B(s)$ in $s=2$. For (\ref{IES33})
$$
\text{Res}_{s=2} \, B(s) \, = \,  \frac{ H(2) \cdot L(1, \chi_{2a} ) \cdot  L(1 +it, \chi_{2a} ) \cdot  \zeta_K (1-it)}{4  \pi \, \Gamma ( 2+l-u) \cdot L( 2, \chi_{2a} )} \hspace{6cm} $$
\begin{equation}\label{IES37}
\hspace{2.8cm} \cdot  \Gamma \big( 1+l-u- \abs{a}  \big) \cdot 
  \Gamma \big( 1+ \abs{a}+it \big) \cdot   \Gamma \big( 1+l-u-\abs{a} -it \big) \cdot     \Gamma \big( 1 +  \abs{a} \big).
\end{equation}

\bigskip
Replacing (\ref{IES37}) in (\ref{F3}) we see that
$$ \Big(  F_{a,a}^l (\psi)  ,  d\epsilon_{it}  \Big) \, = \, O(1) \; + \;  (-1)^{l}  \,
 \frac{ H(2) \cdot L(1, \chi_{2a} ) \cdot  \Gamma ( 1 +  \abs{a} )}{\pi \, L( 2, \chi_{2a} )} \, \bigg[  \sum_{ u=0 }^{ l- \abs{a}}  (-1)^u \; \xi_{a}^l (a,u) \, \frac{  \Gamma ( 1+l-u- \abs{a} ) }{ \Gamma ( 2+l-u) }  \bigg] $$
\begin{equation}\label{IES101}
\cdot \,  \frac{ L(1 +it, \chi_{2a} )}{L(1+it,\chi_0)} 
\, \frac{  \zeta_K (1-it) }{ L(1-it,\chi_{-2a}) } 
\, \frac{  \Gamma ( 1+ \abs{a}+it ) }{ \Gamma(1+it)}
\, \frac{  \Gamma ( 1+l-u-\abs{a} -it ) }{ \Gamma(1+l-u-it) }.
\end{equation}

\bigskip
Applying (\ref{Bo1}), (\ref{Bo4})  and the Stirling's formula in the identity  (\ref{IES101}) we conclude that 
\begin{equation}\label{IES102}
\Big(  F_{a,a}^l (\psi)  ,  d\epsilon_{it}  \Big)  \, = \, O(1).
\end{equation}

% ==================================

\bigskip
\bigskip
For the case $a=-b \neq 0$ and $t \neq 0$ we have that 
$$ B(s) \, = \,  \frac{  H(s) \cdot \zeta_K (\tfrac{s}{2} ) \cdot  \zeta_K (\tfrac{s}{2} +it ) \cdot  L(\tfrac{s}{2}-it, \chi_{-2a} ) \cdot L(\tfrac{s}{2}, \chi_{-2a} ) \cdot  \Gamma \big( \tfrac{s}{2}+l-u \big) \cdot 
  \Gamma \big( \tfrac{s}{2}+it \big) \cdot 
  \Gamma \big( \tfrac{s}{2}+l-u-it \big) \cdot
    \Gamma \big( \tfrac{s}{2} \big)}{\pi^{s} \, \Gamma ( s+l-u) \cdot L( s, \chi_{-2a} )}. $$
    
\bigskip  
As in the previous case $\infty$ is a simple pole of $B(s)$ in  $s=2$. For (\ref{IES33})
\begin{equation}\label{IES42}
\text{Res}_{s=2} \, B(s)  =   \frac{  H(2) \cdot  \zeta_K (1 +it ) \cdot  L(1-it, \chi_{-2a} ) \cdot L(1, \chi_{-2a} ) \cdot  \Gamma \big( 1+l-u \big) \cdot 
  \Gamma \big( 1+it \big) \cdot 
  \Gamma \big( 1+l-u-it \big) } {4\pi \, \Gamma ( 2+l-u) \cdot L( 2, \chi_{-2a} )}.
\end{equation}    

\bigskip
Replacing (\ref{IES42}) in (\ref{F3}) and simplifying
\begin{equation}\label{IES107}
\Big(  F_{a,-a}^l (\psi)  ,  d\epsilon_{it}  \Big) \, = \, O(1) \; + \;  (-1)^{a}  \,  \frac{H(2) \cdot L(1,\chi_{-2a}) }{\pi \, L( 2, \chi_{-2a})} \;
\sum_{ u=0 }^{ l- \abs{a}} \frac{ (-1)^u \, \xi_{a}^l (-a,u) }{1+l-u} \, = \, O(1).
\end{equation}

% =====================================

\bigskip
\bigskip
If $ a \neq \pm b$ and  $ t \neq 0$ it has that $B(s)$ is analytic in $s=2$ and then $ \text{Res}_{s=2} \, B(s) \; =  \; 0$. In this case, for (\ref{F3})
\begin{equation}\label{IES108}
 \Big(  F_{a,b}^l (\psi)  ,  d\epsilon_{it}  \Big) \, = \, O(1).
\end{equation}

\bigskip
The results in (\ref{IES99}), (\ref{IES100}), lemma (\ref{Integral}), (\ref{IES102}), (\ref{IES107}) and (\ref{IES108}) imply the next:
\begin{prop}\label{IncEisSeries}
For   $l=a=b=0$, $ t \neq 0$ we have to 
$$ \Big(  F_{0,0}^0 (\psi)  ,  d\epsilon_{it}  \Big) \, = \, O(1) \; + \; O \Big( \frac{\ln t}{\ln \ln t} \Big) \; + \;   \frac{1}{4 \, \zeta_K(2)} \bigg(  \int_{\Gamma \backslash G} F_{00}^0 (\psi) (g) \, d g  \bigg) \,  \ln t,$$
then
\[
\Big(  F_{0,0}^0 (\psi)  ,  d\epsilon_{it}  \Big) \, \sim  \,  \frac{1}{4 \, \zeta_K(2)} \bigg(  \int_{\Gamma \backslash G} F_{00}^0 (\psi) (g) \, d g  \bigg) \,  \ln t.
\]
In other case and with $ t \neq 0$
$$ \Big(  F_{a,b}^l (\psi)  ,  d\epsilon_{it}  \Big)  \; =  \;   O(1). $$
   
%Para $a=-b \neq 0$
%$$ \Big(  F_{a,-a}^l (\psi)  ,  d\epsilon_{it}  \Big) \, = \,  O(1).$$

%Para $ a \neq \pm b$ 
%$$  \Big(  F_{a,b}^l (\psi)  ,  d\epsilon_{it}  \Big) \, = \,  O(1). $$
\end{prop}

%\newpage
\bigskip
\bigskip
\bigskip
\bigskip
\section{Appendix}

\begin{lemm}\label{SumaEs0}
Let $l \in \Z$, $l \geq 1$, we have 
\[
 \sum_{u=0}^l \frac{(-1)^u \; \xi_{0}^l (0,u) }{1+l-u} \, = \, 0.
\]
\begin{proof}
We  denote by $S$ the previous sum, the following equations are valid
\[
 S   \, = \,  \sum_{u=0}^l \frac{(-1)^u  }{1+l-u} \, \frac{u! (2l-u)!}{l! \, l!} \, \binom{l}{u} \, \binom{l}{u}   \, = \,  \frac{1}{l!} \,  \sum_{u=0}^l (-1)^u  \, \frac{(2l-u)!}{(1+l-u)!}  \, \binom{l}{u}.
\]
Making the change of variable $a=l-u$ and using property $\binom{l}{k} =  \binom{l}{l-k}$ it follows that
\[
 S \, = \,  \frac{(-1)^l}{l!} \,  \sum_{a=0}^l (-1)^a   \, \binom{l}{a} \, \frac{\Gamma(a+l)}{\Gamma(a+1)}.
\]
By formula 0.160 (2) in  Gradshteyn and Ryzhik \cite{GR},
$$ S \, = \,  \frac{(-1)^l}{l!} \,  \frac{\Gamma(l)}{\Gamma(1+l) \; \Gamma(1-l)} \, = \, 0. $$
\end{proof}
\end{lemm}

\begin{lemm}\label{Ap1}
Let $l,k,m \in \tfrac{1}{2} \Z$, $l \geq 0$, $l \equiv k \equiv m \text{ mod } 1$ and  such that $k,m \in [-l,l]$. Moreover, $B \in SU(2)$, $g \in SL(2,\mathbb{C})$. Then
\[
  f_{km}^l(gB,s)   \, = \,  \sum_{\substack{ a=-l \\ a \in \tfrac{1}{2} \Z \\ a \equiv l \text{ mod } 1 }}^{l}  \overline{D^l_{ka} \big(\Phi_{ B^{-1}} \big)} \cdot f_{am}^l(g,s).
\]
In particular, for $g=n[z] a[\lambda] K$ we obtain
\begin{equation}\label{PropE}
 {E}_{km}^l (g,s)   \, = \,   \sum_{\substack{ a=-l \\ a \in \tfrac{1}{2} \Z \\ a \equiv l \text{ mod } 1 }}^{l}   \overline{D^l_{ka} \big( \Phi_{ K^{-1}} \big)}  \cdot E_{am}^l (z+\lambda j,s).
\end{equation}
\end{lemm}

\begin{proof}
We have the following chain of identities
\begin{align}
 f_{km}^l(gB,s)  & \, = \,  \overline{ D_{km}^l \big( \Phi_{ T( gB )^{-1}} \big) } \; \Im \,  gB (j)^{1+s}   &&   \textrm{by (\ref{Deffff})}  \nonumber   \\
  & \, = \,     \overline{ D_{km}^l \big( \Phi_{ B^{-1}} \circ \Phi_{ T(g)^{-1}}  \big) } \; \Im \,  g(j)^{1+s}  
 &&    \nonumber      \\
  & \, = \,    \sum_{\substack{ a=-l \\ a \in \tfrac{1}{2} \Z \\ a \equiv l \text{ mod } 1 }}^{l}  \overline { D^l_{am} \big(\Phi_{ T(g)^{-1}} \big) }
                       \cdot  \overline{D^l_{ka} \big( \Phi_{ B^{-1}} \big)}
                   \; \Im \,  g (j)^{1+s}  &&  \text {by  (\ref{WI10})} \nonumber \\
   & \, = \, \sum_{\substack{ a=-l \\ a \in \tfrac{1}{2} \Z \\ a \equiv l \text{ mod } 1 }}^{l}   \overline{D^l_{ka} \big( \Phi_{ B^{-1}} \big)} \cdot f_{am}^l(g,s).       && \nonumber  
 \end{align}
\end{proof}

\begin{lemm}\label{Ap2}
Let $s \in \mathbb{C}$, $m \in \mathbb{Z}$ such that $m \geq 1$ and  $s+k \neq 0$ for $k=\{0,1, \ldots ,m-1\}$. Then we have 
\[
  \frac{\Gamma'(m+s)}{\Gamma(m+s)} \, = \, \sum_{k=0}^{m-1} \frac{1}{s+k} \, + \, \frac{\Gamma'(s)}{\Gamma(s)}.
\]
\end{lemm}

\begin{proof}

We know that  $\Gamma(1+s) = s  \,  \Gamma(s)$, so
\begin{equation}\label{Le0}
\Gamma(m+s) \, = \,  \prod_{k=0}^{m-1} (s+k) \cdot \Gamma (s).
\end{equation}
Derivating, 
\begin{equation}\label{Le1}
\Gamma'(m+s) \, = \, \Bigg[  \prod_{k=0}^{m-1} (s+k) \Bigg]' \cdot \Gamma (s) \; + \; \Bigg[ \prod_{k=0}^{m-1} (s+k) \Bigg] \cdot \Gamma' (s).
\end{equation}
Dividing the expression in (\ref{Le1}) by (\ref{Le0}) we get the lemma.
\end{proof}

\begin{lemm}\label{Integral}
It is true that
\[
 \int_{\Gamma \backslash G} F_{00}^0 (\psi) (g) \, d g \, = \, H(2).
\]
\end{lemm}
\begin{proof}
We have the following equivalences
\begin{align}
  \int_{\Gamma \backslash G} F_{00}^0 (\psi) (g) \, d g  & \, = \, \int_{\Gamma \backslash G}   \sum_{\sigma \,\in \,_{\scriptstyle\Gamma'_\infty}\backslash\Gamma}   \overline{ D_{00}^0 \big(  \Phi_{ T( \sigma g )^{-1}} \big)} \; \psi \big( \Im \, \sigma g ( j ) \big) \, d g   &&   \textrm{by definition (\ref{DIES})}  \nonumber   \\
  & \, = \,   \int_{\Gamma'_\infty \backslash G}     \psi \big( \Im \,  g ( j ) \big) \, d g  \, = \,    \int_0^\infty    \psi ( \lambda ) \, \int_{\mathbb{R}^2 /\Gamma'_\infty } dxdy \, \int_{SU(2)}  \, d k \, \frac{d\lambda}{\lambda^3}  &&  \nonumber \\
   & \, = \,     \int_0^\infty    \psi ( \lambda ) \,  \frac{d\lambda}{\lambda^3} 
    \, = \,  H(2).     &&  \textrm{by  (\ref{H})} \nonumber 
  \end{align}

%\begin{align}
 % \int_{SU(2)} \overline{ D_{ab}^l \big(  \Phi_{ K^{-1}} \big)} \, dk & \, = \,   \int_{SU(2)}  D_{ab}^l \big(  \Phi_{ K } \big) \, dk &&  \text{ by } (\ref{A8})  \nonumber \\  
  %   & \, = \,  (-1)^{a-b}  \frac{\sqrt{(l+b)!(l-b)!}}{\sqrt{(l+a)!(l-a)!}} \, \int_{SU(2)}  \Phi_{ba}^l (K) \, dk &&  \text{ by } (\ref{A11}) \nonumber \\   
   %          & \, = \,  (-1)^{a-b}  \frac{\sqrt{(l+b)!(l-b)!}}{\sqrt{(l+a)!(l-a)!}} \,   \frac{1}{2l+1} \, \frac{(l+a)!(l-a)!}{(l+b)!(l-b)!} \; \delta_{l,0}  \cdot  \delta_{a,0} \cdot \delta_{b,0} \, = \, \delta_{lab,0}. &&  \label{Int2} 
 %\end{align}  

%Substituyendo (\ref{Int1}) en (\ref{Int2})
%$$  \int_{\Gamma / G} F_{ab}^l (\psi) (g) \, d g \, = \,
 %\abs{\Lambda_D} \cdot \delta_{lab,0} \cdot  \int_0^\infty    \psi ( \lambda ) \, \frac{d\lambda}{\lambda^3} \, = \,
 %\abs{\Lambda_D} \cdot \delta_{lab,0} \cdot H(2). $$
\end{proof}

\begin{lemm} \label{LFC} The following identity is valid
$$  \sum_{0 \neq  w \in Z[i]}  c(w) \cdot \abs{w}^s \cdot  \bigg( \frac{ w }{\abs{w}}  \bigg)^{\alpha} \cdot \sigma_{\nu} (w,0) \, = \,  \frac{1}{4}
 \, \frac{L(-\tfrac{s}{2}, F_{pq}^l, \chi_\alpha ) \, L(-\tfrac{s}{2}-\nu, F_{pq}^l, \chi_\alpha ) }{ L( -s-\nu, \chi_{2\alpha}) } .   $$
\end{lemm}
\begin{proof}

Let $$R(s)  :=  \sum_{0 \neq  w \in Z[i]}  c(w) \cdot \abs{w}^s \cdot  \bigg( \frac{ w }{\abs{w}}  \bigg)^{\alpha} \cdot \sigma_{\nu} (w,0)  = \prod_{ \substack{ p \in Z[i] \\  p  \text{ primo}  } }  R_p (s),$$
with
$$ R_p (s) \, = \, \sum_{j=0}^\infty   c(p^j) \cdot \abs{p^j}^s \cdot  \bigg( \frac{ p^j }{\abs{p^j}}  \bigg)^{\alpha} \cdot \sigma_{\nu} (p^j,0). $$

\bigskip
We observed that
$$ \sigma_{\nu} (p^j,0) \, = \, \frac{1}{4} \; \sum_{t=0}^j  \abs{p^t}^{2 \nu}  \, = \, \frac{1}{4} \; \sum_{t=0}^j \Big( \abs{p}^{2\nu}  \Big)^{t} 
\, = \, \frac{1}{4} \cdot \frac{1-\abs{p}^{2\nu(j+1)}}{1-\abs{p}^{2\nu}}.  $$
% $$ = \; \frac{1}{4} \cdot \frac{1-\Big( \abs{p}^{2\nu}  \Big)^{j+1} }{1-\Big( \abs{p}^{2\nu}  \Big) } \, = \, \frac{1}{4} \cdot \frac{1-\abs{p}^{2\nu(j+1)}}{1-\abs{p}^{2\nu}}.  $$

Therefore,  
 $$  R_p (s) \, = \,  \frac{1}{4} \; \sum_{j=0}^\infty   c(p^j) \cdot \abs{p}^{js} \cdot  \frac{ p^{j\alpha} }{\abs{p}^{j\alpha}}  \cdot \frac{1-\abs{p}^{2\nu(j+1)}}{1-\abs{p}^{2\nu}} \, = \,  \frac{1}{4} \, \frac{ 1}{ 1-\abs{p}^{2\nu} }  \; \sum_{j=0}^\infty   c(p^j) \cdot  \frac{ p^{j\alpha} }{\abs{p}^{j(-s+\alpha)}}  \cdot \Big( 1-\abs{p}^{2\nu(j+1)} \Big) $$

$$ = \, \frac{1}{4} \,  \frac{ 1}{ 1-\abs{p}^{2\nu} }  \; \Bigg[ \sum_{j=0}^\infty  c(p^j) \Big( p^{\alpha} \cdot \abs{p}^{s-\alpha} \Big)^j \; - \; \abs{p}^{2\nu} \; \sum_{j=0}^\infty c(p^j)  \Big(  p^{\alpha} \cdot \abs{p}^{s-\alpha+2\nu} \Big)^j \Bigg] $$

$$ = \, \frac{1}{4} \, \frac{1}{ 1-\abs{p}^{2\nu} }  \; \Bigg[
\frac{1}{1 - c(p)\cdot p^{\alpha} \, \abs{p}^{s-\alpha} + 
 p^{2\alpha} \, \abs{p}^{2s-2\alpha}} \; - \; \frac{\abs{p}^{2\nu}}{1 - c(p) \cdot p^{\alpha} \, \abs{p}^{s-\alpha+2\nu} + 
 p^{2\alpha} \, \abs{p}^{2s-2\alpha+4\nu}} \Bigg] $$

%$$ = \, \frac{1}{4} \, \frac{1}{ 1-\abs{p}^{2\nu} }  \; \Bigg[
%\frac{ 1 - c(p) \cdot p^{\alpha} \, \abs{p}^{s-\alpha+2\nu} + 
% p^{2\alpha} \, \abs{p}^{2s-2\alpha+4\nu} -  \abs{p}^{2\alpha} + c(p)\cdot p^{\alpha} \, \abs{p}^{s-\alpha+2\nu} - 
% p^{2\alpha} \, \abs{p}^{2s-2\alpha+2\nu}  }{ \big( 1 - c(p)\cdot p^{\alpha} \, \abs{p}^{s-\alpha} + 
% p^{2\alpha} \, \abs{p}^{2s-2\alpha} \big) \, \big( 1 - c(p) \cdot p^{\alpha} \, \abs{p}^{s-\alpha+2\nu} + 
% p^{2\alpha} \, \abs{p}^{2s-2\alpha+4\nu} \big) }
%\Bigg] $$

%$$ = \, \frac{1}{4} \,  \frac{ 1}{ 1-\abs{p}^{2\nu} }  \; 
%\frac{ \big( 1 -\abs{p}^{2\nu} \big) \, \big( 1 - p^{2\alpha} \, \abs{p}^{2s-2\alpha+2\nu}   \big)  }{ \big( 1 - c(p)\cdot p^{\alpha} \, \abs{p}^{s-\alpha} +  p^{2\alpha} \, \abs{p}^{2s-2\alpha} \big) \, \big( 1 - c(p) \cdot p^{\alpha} \, \abs{p}^{s-\alpha+2\nu} + 
% p^{2\alpha} \, \abs{p}^{2s-2\alpha+4\nu} \big) }
% $$

$$ = \,  \frac{1}{4} \, 
\frac{  1 - p^{2\alpha} \, \abs{p}^{2s-2\alpha+2\nu}   }{ \big( 1 - c(p)\cdot p^{\alpha} \, \abs{p}^{s-\alpha} +  p^{2\alpha} \, \abs{p}^{2s-2\alpha} \big) \, \big( 1 - c(p) \cdot p^{\alpha} \, \abs{p}^{s-\alpha+2\nu} + 
 p^{2\alpha} \, \abs{p}^{2s-2\alpha+4\nu} \big) }.
 $$

\end{proof}

\bigskip
\subsection*{Acknowledgements}
The author was supported by CINVESTAV with a posdoctoral scholarship, via FORDECYT 265667 project. I want to  thank to Juan Manuel Burgos and  Jacob Mostovoy for the support to get that scholarship.

% Both authors would like to thank Professor Paul Garrett for his comments and references regarding Eisenstein series.

\end{document}